\pdfoutput=1
\RequirePackage{ifpdf}
\ifpdf 
\documentclass[pdftex]{sigma}
\else
\documentclass{sigma}
\fi

\usepackage{tikz,tikz-cd,mathrsfs,dsfont,tensor}

\newcommand{\N}{\mathbb{N}}

\newcommand{\K}{\Bbbk} 
\newcommand{\aff}{\mathbb{A}} 

\newcommand{\mc}[1]{\mathcal{#1}}
\newcommand{\mf}[1]{\mathfrak{#1}}
\newcommand{\mr}[1]{\mathrm{#1}}
\newcommand{\ms}[1]{\mathsf{#1}}
\newcommand{\mx}[1]{\mathscr{#1}}

\newcommand{\1}{\mathds{1}} 
\newcommand{\op}{\mathrm{op}} 

\newcommand{\Stk}{\mathsf{St}_\K}
\newcommand{\Algst}{\mathsf{AlgSt}_\K}
\newcommand{\bal}{\mathcal{B}^\mathrm{bal}_{0,n}(G)}
\newcommand{\balk}{\mathcal{B}^\mathrm{bal}_{0,k}(G)}
\newcommand{\balN}{\mathcal{B}^\mathrm{bal}_{0,N}(G)}
\newcommand{\balNsm}{\mathcal{B}^\mathrm{sm}_{0,N}(G)}
\newcommand{\balsm}{\mathcal{B}^\mathrm{sm}_{0,n}(G)}
\newcommand{\balc}{\mathcal{B}^\mathrm{bal}_{0,\boldsymbol{c}}(G)}

\newcommand{\balmv}{\mathcal{B}^\mathrm{bal}_{0,\boldsymbol{\mf{m}}(v)}(G)}

\newcommand{\balmvsm}{\mathcal{B}^\mathrm{sm}_{0,\boldsymbol{\mf{m}}(v)}(G)}
\newcommand{\balmwsm}{\mathcal{B}^\mathrm{sm}_{0,\boldsymbol{\mf{m}}(w)}(G)}

\newcommand{\Kstk}{K_0(\Algst)}
\newcommand{\Ksnmod}{K_0^{\boldsymbol{S}_n}({\Algst})}
\newcommand{\Ksmod}{K_0^{\ms{S}}({\Algst}/ B)}
\newcommand{\KsmodK}{K_0^{\ms{S}}({\Algst})}
\newcommand{\Ksmodr}{K_0^{\ms{S}, r}({\Algst}/B)}
\newcommand{\BG}{{\mathcal{B}G}}
\newcommand{\cyinst}{\overline{\mathcal{I}}_{\boldsymbol{\mu}}(\BG)}

\numberwithin{equation}{section}

\newtheorem{Theorem}{Theorem}[section]
\newtheorem*{Theorem*}{Theorem}
\newtheorem{Corollary}[Theorem]{Corollary}
\newtheorem{Lemma}[Theorem]{Lemma}
\newtheorem{Proposition}[Theorem]{Proposition}
 { \theoremstyle{definition}
\newtheorem{Definition}[Theorem]{Definition}

\newtheorem{Example}[Theorem]{Example}
\newtheorem{Remark}[Theorem]{Remark}
\newtheorem{Notation}[Theorem]{Notation} }

\begin{document}
\allowdisplaybreaks

\renewcommand{\thefootnote}{}

\newcommand{\arXivNumber}{2212.14722}

\renewcommand{\PaperNumber}{107}

\FirstPageHeading

\ShortArticleName{On the Motivic Class of the Moduli Stack of Twisted $G$-Covers}

\ArticleName{On the Motivic Class of the Moduli Stack\\ of Twisted $\boldsymbol{G}$-Covers\footnote{This paper is a~contribution to the Special Issue on Enumerative and Gauge-Theoretic Invariants in honor of Lothar G\"ottsche on the occasion of his 60th birthday. The~full collection is available at \href{https://www.emis.de/journals/SIGMA/Gottsche.html}{https://www.emis.de/journals/SIGMA/Gottsche.html}}}

\Author{Massimo BAGNAROL~$^{\rm a}$ and Fabio PERRONI~$^{\rm b}$}

\AuthorNameForHeading{M.~Bagnarol and F.~Perroni}

\Address{$^{\rm a)}$~Dipartimento di Ingegneria e Architettura, Universit\`a degli Studi di Trieste,\\
\hphantom{$^{\rm a)}$}~via Valerio 6/1, 34127 Trieste, Italy}
\EmailD{\href{mailto:massimo.bagnarol@dia.units.it}{massimo.bagnarol@dia.units.it}}

\Address{$^{\rm b)}$~Dipartimento di Matematica e Geoscienze, Universit\`a degli Studi di Trieste,\\
\hphantom{$^{\rm b)}$}~via Valerio 12/1, 34127 Trieste, Italy}
\EmailD{\href{mailto:fperroni@units.it}{fperroni@units.it}}

\ArticleDates{Received January 02, 2023, in final form December 19, 2023; Published online December 27, 2023}

\Abstract{We describe the class, in the Grothendieck group of stacks, of the stack of twisted $G$-covers of genus $0$ curves, in terms of the loci corresponding to covers over smooth bases.}


\Classification{14D23; 14F45; 14H30}

\renewcommand{\thefootnote}{\arabic{footnote}}
\setcounter{footnote}{0}

\section{Introduction}
Moduli stacks of twisted stable maps have been introduced in the framework of Gromov--Witten theory
for orbifolds and for Deligne--Mumford stacks,
independently by Chen--Ruan \cite{CR02} and by Abramovich--Graber--Vistoli \cite{AGV08, AV02}. In the case where the target is the classifying stack $\BG$ of a~finite group~$G$,
a twisted stable map is called a twisted $G$-cover~\cite{ACV03}.
The moduli stack ${\mathcal{B}^\mathrm{bal}_{g,n}(G)}$ of balanced twisted stable maps to $\BG$ is isomorphic to
the stack of admissible $G$-covers, which has been extensively studied due to its several applications
(see, e.g., \cite{ACV03, FG03,HM82,JKK05, PJ95}).

The present
work originates from \cite{CLP11, CLP15,CLP16}
where the authors studied the classification problem for the connected components
of moduli spaces of Galois covers of smooth curves. The results obtained there give a complete description
of such components for any genus of the base curves when the Galois group is the dihedral group,
in the general case one only has a complete classification when the genus of the base curves is large enough.
In this article, we consider the compactifications of moduli spaces of Galois covers of curves
given by admissible covers, restricting our attention to the case where the genus of the base curves is $0$.
The aim is to investigate how the boundary can be used to obtain information about
the geometry of the moduli spaces.
Our main result (Theorem~\ref{mainthm}) gives a description of the class of $\bal$
in the Grothendieck group of stacks, i.e., of the \textit{motivic class} of $\bal$.
This description is obtained by stratifying the moduli stack as the disjoint union of locally closed
substacks $\mathcal B^{\rm bal}(G, \tilde \tau)$, which correspond to the $G$-covers
of a given combinatorial structure $\tilde \tau$, where $\tilde \tau$ is a gerby tree (cf.\ Definition \ref{gerbytree}).
The main technical step is Proposition \ref{mainprop}, where we show how these strata are
obtained by gluing together the loci in $\bal$ corresponding to covers over smooth bases.
In order to keep track of the actions of the symmetric groups on the marked points, these
gluing constructions correspond to two basic operations, the \textit{Day convolution}
(Definition \ref{Day}) and the \textit{composition} (Definition \ref{composition}).
As a result, we obtain a formula \eqref{maineq} that
expresses the motivic class of $\bal$
recursively in terms of
the loci $\mc{B}_k$, for $k\leq n$, corresponding to covers over smooth bases
(cf.\ Remark~\ref{remmainthm}).
Here we follow the strategy of \cite{GP06} (where the authors compute the Betti numbers of the moduli spaces
of stable maps to projective spaces), further extended in~\cite{BagnarolMN} to the case of
moduli spaces of stable maps to Grassmannians.
Over the field of complex numbers, by applying the Poincar\'e (or Serre) characteristic
\cite{BD07,GP06},
our formula allows to compute the Betti numbers of $\bal$ in terms of
the cohomology of $\mc{B}_k$, as a representation of the symmetric group $\boldsymbol{S}_k$, $k=3, \dots ,n$,
and the decomposition of $\mc{B}_k$ as union of the open and closed substacks
$\mc{B}_{\boldsymbol{c}} := \balc \cap \mc{B}_k$,
where $\boldsymbol{c}=(c_1, \dots , c_k)$ varies among the sequences of~$k$ conjugacy classes of~$G$,
and $\balc$ is the full sub-category of $\balk$ whose objects are the twisted $G$-covers
whose evaluation at the marked points is equal to $\boldsymbol{c}$ (cf.\ Definition~\ref{balc}).
Notice that $\mc{B}_{\boldsymbol{c}}$ is determined by $\mc{B}_k$, as a stack over
$\bar{\mathcal I}_{\boldsymbol{\mu}} (\BG)^k$ with respect to the evaluation morphism.
We will address the problem of computing effectively these invariants
of $\mc{B}_n$ in future works.

The structure of the article is as follows. In Section~\ref{recall}, we recall the definition and collect
basic properties of $\bal$. In Section~\ref{section3}, we define the Day convolution and the composition, and
we prove their main properties that will be used. In Section~\ref{section4}, following the definition of the Grothendieck group of
stacks from \cite{Eke09}, we introduce the Grothendieck group
of $\ms{S}$-modules, over which there is an induced composition
operation.
In Section~\ref{section5}, we define the notion of gerby tree~$\tilde \tau$ (following~\cite{AJT16}) and we show that the locus
in $\bal$ of covers of a fixed combinatorial type~$\tilde \tau$ is locally closed.
In Section~\ref{section recursive relations}, we prove our main results, Proposition~\ref{mainprop} and Theorem~\ref{mainthm}.
In the last Section~\ref{section7}, we show how these results can be used to compute the Hodge--Grothendieck
characteristic of $\bal$ in terms of $\mathcal{B}_k$ for $k=3, \dots , n$.

\subsection*{Conventions}
The symbol $\N$ denotes the set of nonnegative integers.
For each $n \in \N$, the symmetric group on~$n$ letters is denoted by $\boldsymbol{S}_n$.

We fix a universe, and we use the term ``small'' in reference to the chosen universe.
As usual, the category of small categories is denoted by $\ms{Cat}$, the category of small sets is denoted by $\ms{Set}$,
and $\boldsymbol{1}$ denotes the terminal category.
For all (enriched) categories $\ms{C}$ and $\ms{D}$,
$\ms{D}^\ms{C}$ denotes the (enriched) category of (enriched) functors from $\ms{C}$ to $\ms{D}$.
A $2$-category is a $\ms{Cat}$-enriched category, which is sometimes referred to as a strict $2$-category.

In order to avoid set-theoretic issues when considering toposes, by a stack over a base scheme~$T$
we mean a stack in groupoids over a suitable small sub-site of the big fppf site of~$T$ (constructed as in \cite[Section~IV.2.5]{SGA4IV}).

 For what concerns group actions on stacks, we follow \cite{Romagny}, in particular, according to \cite[Definition~2.1]{Romagny}
the actions that we will consider are all strict. This degree of generality is sufficient for our purposes, even though in several places
one could allow more general actions.

\section[The stack of twisted G-covers]{The stack of twisted $\boldsymbol{G}$-covers}\label{recall}

In this section, we recall the definition of the stack of balanced twisted $G$-covers of $n$-pointed genus $0$ curves,
$\bal$, following \cite{ACV03},
to which we refer for further details and properties of this stack.

Let $\K$ be a fixed algebraically closed field of characteristic $0$. By an algebraic $\K$-stack we mean an algebraic stack of finite type over $\K$ which has affine stabilizers.
The $2$-category of algebraic $\K$-stacks is denoted by $\Algst$, while $\Stk$ denotes the $2$-category of stacks over $\K$.

Let $G$ be a finite group. Our aim is to study the class of $\bal$
in the Grothendieck ring of algebraic $\K$-stacks $\Kstk$.

Roughly speaking, the objects of $\bal$ are twisted stable maps from twisted curves to $\BG$, the classifying stack of $G$.

\begin{Definition}\label{def twisted curve}
A \textit{twisted curve} over a $\K$-scheme $T$ is a morphism $\mathcal{C} \to T$, where
\begin{enumerate}\itemsep=0pt
\item[(1)] $\mathcal{C}$ is a Deligne--Mumford stack;
\item[(2)] $\mathcal{C} \to T$ is \'etale locally isomorphic to $[U/{\boldsymbol{\mu}}_r] \to \operatorname{Spec} (A)$, and $U$ is
\begin{itemize}\itemsep=0pt
\item either $\operatorname{Spec} (A[z])$, with ${\boldsymbol{\mu}}_r$ acting as $(\zeta , z) \mapsto \zeta \cdot z$,
\item or $\operatorname{Spec}(A[z,w]/(zw-t))$, for some $t\in A$, with ${\boldsymbol{\mu}}_r$ acting as $(\zeta , (z,w)) \mapsto (\zeta \cdot z , \zeta^{a}\cdot w)$,
for some $a\in (\mathbb{Z}/r\mathbb{Z})^\times$.
\end{itemize}
\end{enumerate}
Notice that in the second case the fiber over $t=0$ has a node in $U$, which is locally smoothable if and only if $a\equiv -1$ (mod $r$).
A twisted curve is called \textit{balanced} if at every node it is presented as above with $a\equiv -1$ (mod $r$).
\end{Definition}

\begin{Definition}\label{twnodnptdcurve}
Let $N$ be a finite set. A \textit{twisted nodal $N$-pointed curve} over a~$\K$-scheme $T$ is a~diagram
\[
\begin{tikzcd}
\Sigma \arrow[r, hook] & \mathcal{C} \arrow{d}\\
& C \arrow{d} \\
& \,T,
\end{tikzcd}
\]
where
\begin{itemize}\itemsep=0pt
\item[(1)] $\mathcal{C} \to T$ is a proper twisted curve over $T$;
\item[(2)] $\Sigma =\bigcup_{i\in N} \Sigma_i$, where the $\Sigma_i \subset \mathcal{C}$, for $i\in N$, are disjoint closed substacks
contained in the smooth locus of $\mathcal{C} \to T$;
\item[(3)] for every $i\in N$, $\Sigma_i \to T$ is an \'etale gerbes (see \cite[Section~3.3]{AGV08} for a definition);
\item[(4)] the morphism $\mathcal{C} \to C$ exhibits $C$ as the coarse moduli space of $\mathcal{C}$;
\item[(5)] for every $i\in N$, let $|\Sigma_i| \subset C$ be the image of $\Sigma_i$ (called the $i$-th marking
in $C$),
and let $C_{\rm gen}$ the complement of the nodes and the markings in $C$,
then the morphism $\mathcal{C} \to C$ is an isomorphism over $C_{\rm gen}$.
\end{itemize}
A twisted nodal $N$-pointed curve over $T$ is called \textit{balanced} if $\mathcal{C} \to T$ is a balanced twisted curve.

For a positive integer $n$, by a (twisted nodal) $n$-pointed curve we mean a (twisted nodal) $N$-pointed curve
where $N=\{ 1, \dots , n\}$.
\end{Definition}

\begin{Remark}\label{remark on twisted curves}
In the previous definition, one should think of $\Sigma$ as the marked points of the twisted curve, indeed we will refer to
$\Sigma_i$
 as the $i$-th marked point. The ``gerby" structure of the marked points arise naturally from Definition \ref{def twisted curve},
according to which, \'etale locally, $\Sigma_i \to T$ is isomorphic to $[\operatorname{Spec} (A[z]/(z))/ {\boldsymbol{\mu}}_r] \to \operatorname{Spec} (A)$, where
${\boldsymbol{\mu}}_r$ acts trivially.
The integer $r$ is called the \textit{index} of the marked point over $\operatorname{Spec} (A)$.
If $\Sigma_i$ has constant index $r$, then it is canonically banded by ${\boldsymbol{\mu}}_r$ \cite[Theorem~4.2.1]{AGV08}.
We refer to \cite[Section~4]{AGV08} and \cite{Olsson} for more details on twisted curves and their moduli.
\end{Remark}

Consider now the classifying stack $\BG$ of $G$, that is the category whose objects over a $\K$-scheme $T$ are principal $G$-bundles $P\to T$,
and morphisms are $G$-equivariant fiber diagrams of such bundles:
\[
\begin{tikzcd}
P_1 \arrow{r} \arrow{d} & P_2 \arrow{d} \\
T_1 \arrow{r} & \,T_2.
\end{tikzcd}
\]
$\BG$ is a global quotient stack with presentation $\BG=[ \operatorname{Spec} (\K) /G ]$ and coarse moduli space $\operatorname{Spec} (\K)$.

\begin{Definition}\label{twiste_Gcover_Npointed_curve}
Let $N$ be a finite set, and let $T$ be a $\K$-scheme.
A \textit{twisted $G$-cover of an $N$-pointed curve of genus $g$} over $T$ consists of a commutative diagram
\[
\begin{tikzcd}
\Sigma\arrow[r, hook] & \mathcal{C} \arrow{r} \arrow{d} & \BG \arrow{d} \\
& C \arrow{r} \arrow{d} & \operatorname{Spec} (\K) \\
& T &
\end{tikzcd}
\]
such that the following conditions are satisfied:
\begin{itemize}\itemsep=0pt
\item[(i)] $(\mathcal{C}\to C \to T, \Sigma)$ is a twisted nodal $N$-pointed curve over $T$;
\item[(ii)] for any $i\in N$ and for $|\Sigma_i |$ being the image of $\Sigma_i$ under the morphism $\mathcal{C} \to C$,
then $(C\to T, (|\Sigma_i|)_{i\in N})$
is a stable $N$-pointed curve of genus $g$;
\item[(iii)] the morphism $\mathcal{C} \to \BG$ is representable.
\end{itemize}
A twisted $G$-cover is called \textit{balanced} if the corresponding twisted curve is balanced.
\end{Definition}
\begin{Remark}
To justify the name``twisted $G$-cover" in the previous definition, notice that the morphism $\mathcal{C} \to \BG$
corresponds to a principal $G$-bundle $P\to \mathcal{C}$. Moreover, the representability condition~(iii) implies that
$P\to T$ is a projective nodal curve (see \cite[Lemma~2.2.1]{ACV03} for details), so $P\to C$ is a $G$-cover in the usual sense,
which is ramified (at most) over the marked points~$|\Sigma_i|$, $i\in N$, and the nodes of~$C$.
\end{Remark}

We have now introduced all the objects needed to define the stack we are interested in.
Consider the $2$-category whose objects over a $\K$-scheme $T$ are balanced twisted $G$-covers of an
$N$-pointed curve of genus $0$ over $T$,
$1$-morphisms are given by fiber diagrams, and {$2$-morphisms are given by base-preserving natural transformations}.
As explained in \cite[Section~2.1.5]{ACV03}, this $2$-category is equivalent to a category, obtained by replacing $1$-morphisms by their
$2$-isomorphism classes (see \cite[Proposition~4.2.2]{AV02} for the proof).
The category of balanced twisted $G$-covers thus obtained is denoted $\balN$, respectively
$\bal$ when $N=\{1,\dots, n\}$.

The following result is a special case of Theorems~2.1.7,~3.0.2, and of Corollary 3.0.5 in \cite{ACV03}
(see also~\cite{AV02}).
\begin{Theorem}
For any finite set $N$ of cardinality $\#N\geq 3$, $\balN$ is a smooth Deligne--Mumford stack of dimension $\#N-3$
and with projective coarse moduli space.
There is a morphism $\balN \to \overline{\mathcal M}_{0,N}$
which is flat, proper and quasi-finite~$(\overline{\mathcal M}_{0,N}$ is the moduli space of stable $N$-pointed curves of genus~$0)$.
\end{Theorem}

\begin{Remark}
In general $\mathcal B^{\rm bal}_{g,N}(G)$ is defined as the stack of balanced twisted $G$-covers of
$N$-pointed curves of genus $g$.
As proved in \cite[Theorem~4.3.2]{ACV03}, this stack is isomorphic to the stack of admissible $G$-covers $\mathcal{A}dm_{g,N}(G)$.
Throughout this article, it is useful to have in mind both points of view.
\end{Remark}

\begin{Notation}\label{balsm}
We will denote with $\balNsm$ the locus in $\balN$ of twisted $G$-covers of smooth $N$-pointed curves of genus $0$,
i.e., where the curves $\mathcal{C}$ in Definition \ref{twiste_Gcover_Npointed_curve} are smooth.
Notice that $\balNsm$ is an open sub-stack of $\balN$.
\end{Notation}

The stack $\balN$ comes equipped with evaluation maps
${\rm ev}^i_N$, for $i\in N$ (respectively ${\rm ev}^i_n$, if $N=\{ 1, \dots , n\}$), defined as follows.
Let $(f\colon \mathcal{C} \to \BG, \Sigma)$ be an object of $\balN$ over $T$,
the restriction of $f$ to $\Sigma_i$ yields the diagram
\begin{equation}\label{ev}
\begin{tikzcd}
\Sigma_i \arrow{r}{f_{|\Sigma_i}} \arrow{d} & \BG \\
\,T. &
\end{tikzcd}
\end{equation}
If $\Sigma_i$ has constant index~$r$, \eqref{ev} is an object of $\bar{\mathcal I}_{{\boldsymbol{\mu}}_r} (\BG)$ over $T$,
whose objects are, by definition, pairs $(\mathcal G, \phi)$ of a gerbe $\mathcal G$
banded by ${\boldsymbol{\mu}}_r$ and a representable morphism $\phi \colon \mathcal G \to \BG$
(as noticed in Remark~\ref{remark on twisted curves}, $\Sigma_i$ is banded by ${\boldsymbol{\mu}}_r$).

 Let $\bar{\mathcal I}_{\boldsymbol{\mu}} (\BG):= \sqcup_r \bar{\mathcal I}_{{\boldsymbol{\mu}}_r} (\BG)$,
it is called the \textit{stack of cyclotomic gerbes} in $\BG$ \cite[Definition~3.3.6]{AGV08}.
Then
\[
{\rm ev}^i_N \colon \ \balN \to \bar{\mathcal I}_{\boldsymbol{\mu}} (\BG)
\]
maps the object $(f\colon \mathcal{C} \to \BG, \Sigma)$ to $(\Sigma_i, f_{|\Sigma_i})$.

Notice that, as usual, the group $\boldsymbol{S}_N$ of bijections $N\to N$ (respectively
the symmetric group~$\boldsymbol{S}_n$, if $N=\{ 1, \dots , n\}$) acts naturally on $\balN$ by permutation of the marked points
(and this action is strict in the sense of \cite[Definition~1.3]{Romagny}).
Furthermore, defining
\[
{\rm ev}_N := \prod_{i\in N} {\rm ev}^i \colon \ \balN \to \bar{\mathcal I}_{\boldsymbol{\mu}} (\BG)^N ,
\]
we obtain a strict $\boldsymbol{S}_N$-morphism, where $\boldsymbol{S}_N$ acts on the target also by permutation.

\subsection{The coarse moduli space of the stack of cyclotomic gerbes}\label{recall_sub}
For lack of a suitable reference, we show here that the coarse moduli space of ${\overline{\mathcal{I}}_{\boldsymbol{\mu}}(\BG)}$,
which will be denoted ${\overline{I}_{\boldsymbol{\mu}}(\BG)}$ in the article, can be identified with the set of conjugacy classes of $G$.

To this aim, by \cite[Proposition~3.4.1]{AGV08} we have that ${\overline{\mathcal{I}}_{\boldsymbol{\mu}_r}(\BG)}$ is isomorphic to the
rigidification of the cyclotomic inertia stack
${{\mathcal{I}}_{\boldsymbol{\mu}_r}(\BG)}$ with respect to $\boldsymbol{\mu}_r$, hence its coarse moduli space is isomorphic to that of
${{\mathcal{I}}_{\boldsymbol{\mu}_r}(\BG)}$.
By \cite[Definition~3.1.1]{AGV08}, an object of ${{\mathcal{I}}_{\boldsymbol{\mu}_r}(\BG)}$ over a scheme $T$ is a pair $(\xi, \alpha)$,
where $\xi$ is a principal $G$-bundle over $T$ and $\alpha \colon (\boldsymbol{\mu}_r)_T \to \operatorname{Aut}_T (\xi)$ is an injective homomorphism.
An arrow from $(\xi, \alpha)$ over $T$ to $(\xi', \alpha')$ over $T'$ is an arrow $F\colon \xi \to \xi'$ such that the following diagram commutes:
\[
\begin{tikzcd}
(\boldsymbol{\mu}_r)_T \arrow{d}{\alpha} \arrow{r} & (\boldsymbol{\mu}_r)_{T'} \arrow{d}{\alpha'} \\
 \operatorname{Aut}_T (\xi) \arrow{d} & \arrow{l} \operatorname{Aut}_{T'} (\xi') \arrow{d} \\
 T \arrow {r} & \,T',
\end{tikzcd}
\]
where $(\boldsymbol{\mu}_r)_T \to (\boldsymbol{\mu}_r)_{T'}$ is induced by $T\to T'$ and ${\rm Id}_{\boldsymbol{\mu}_r}$,
 $\operatorname{Aut}_T (\xi) \leftarrow \operatorname{Aut}_{T'} (\xi')$ is the map induced by $F$ (it associates any $\phi' \in \operatorname{Aut}_{T'} (\xi')$ with the unique
$\phi \in \operatorname{Aut}_T (\xi)$ such that $F\circ \phi = \phi' \circ F$).
Notice that, if we fix a generator of $\boldsymbol{\mu}_r$, $\alpha$ corresponds to an element $\phi \in \operatorname{Aut}_T (\xi)$ of order $r$.

Over $\operatorname{Spec} (\K)$, every principal $G$-bundle $\xi$ is isomorphic to the trivial one, which we denote $\xi_0$.
Furthermore, we can identify $\operatorname{Aut} (\xi_0)$ with $G$ by associating any $g\in G$ with the automorphism $\phi_g \colon h \mapsto hg$.
An arrow $(\xi_0, \phi_{g_1}) \to (\xi_0, \phi_{g_2})$ is given by an automorphism $\phi_g\colon \xi_0 \to \xi_0$
such that $\phi_{g_1} = \phi_g^{-1} \circ \phi_{g_2} \circ \phi_g$.
From this description, we see that
\[
{{\mathcal{I}}_{\boldsymbol{\mu}_r}(\BG)}(\operatorname{Spec} (\K))/\cong \leftrightarrow \{ (\xi_0 , \phi_g) \mid g\in G \, \mbox{of order} \, r \}/\cong
 \leftrightarrow \{ g\in G \mid g \, \mbox{has order} \, r \}/G,
\]
where the quotient to the right is with respect to the action by conjugation.

We use this description to define certain open and closed loci in $\balN$, as follows.
\begin{Definition}\label{balc}
Let $\boldsymbol{c}=(c_i)_{i\in N}$ be a sequence of conjugacy classes of $G$,
i.e., $\boldsymbol{c} \in {\overline{I}_{\boldsymbol{\mu}}(\BG)}^N$.
Let $\balc$ be the full sub-category of $\balN$ whose objects are the twisted $G$-covers
$(\mathcal{C} \to T, f\colon \mathcal{C} \to \BG, \Sigma)$ of $\balN$ such that, over every geometric point $\Omega \in T$,
\[
{\rm ev}_N (\mathcal{C} \to T, \, f\colon \mathcal{C} \to \BG, \Sigma)_{|\Omega} = \boldsymbol{c} .
\]
As explained in \cite{AV02} at the end of Section~8.5, this last condition is locally constant on $T$, hence
$\balc$ is an open and closed sub-stack of $\balN$.
\end{Definition}

\section[The category of S-modules over a base]{The category of $\boldsymbol{\ms{S}}$-modules over a base}\label{section3}

In order to study the class of $\bal$ in $\Kstk$ (actually, in a slightly different ring that we will define later),
we need to translate the combinatorics of its boundary into relations within this ring.
As usual the boundary of $\bal$ corresponds to twisted $G$-covers $(f\colon \mathcal{C} \to \BG, \Sigma)$ such that
$\mathcal{C}$ is reducible,
say $\mathcal{C}= \mathcal{C}' \cup \mathcal{C}''$ with $\mathcal{C}'$ and $\mathcal{C}''$
being union of irreducible components of $\mathcal{C}$.
The restriction of $f$ to $\mathcal{C}'$ and $\mathcal{C}''$ yields twisted $G$-covers $(f' \colon \mathcal{C}' \to \BG , \Sigma')$
and $(f'' \colon \mathcal{C}'' \to \BG , \Sigma'')$ with fewer marked points. So it is expected that the class of the boundary of $\bal$
can be related to the classes of ${\mathcal{B}^\mathrm{bal}_{0,k}(G)}$ with $k<n$. This relation is the content of
our main results, Proposition~\ref{mainprop} and Theorem~\ref{mainthm}.
In this section, we develop a formalism that is needed when we express $f$ as the gluing of $f'$ and $f''$.
One complication arises here because the marked points $\Sigma_1, \dots , \Sigma_n$ are ``stacky points'',
which contain the information of the local monodromies of the $G$-cover $P\to C$.
Furthermore, we shall also take into account the stacks $\bal$, with
 the action of the symmetric group $\boldsymbol{S}_n$, altogether as $n\geq 3$
(this is, Example \ref{exSmod}).

An approach towards this direction can be recovered from the theory of colored operads, where many of the tools that we
use originated.
The connection with this theory has been already highlighted in \cite{Pet13}.
Here we adopt the point of view of~\cite{Kel05} according to which an operad is a monoid for a certain tensor product $\circ$
(the so called ``composition'', which in our context is defined in Definition~\ref{composition})
in a particular category of functors. However, we can not use directly the results therein for two reasons:
firstly, the category of stacks is a $2$-category, hence we need to work in the context of enriched categories;
secondly, as mentioned above, the gluing maps are more complicated in our setting due to the local monodromies.
Let us recall the main definitions and results, which we adapt for our purposes.
Since these results are effortlessly reworked versions of well-known facts in the non-enriched case, we do not present complete proofs, but we rather indicate which classical results they are adapted from.

We consider categories enriched over a fixed cosmos $(\mx{V},\otimes,\1)$, i.e., a symmetric monoidal closed category which is bicomplete. For simplicity, we assume that the unit $\1$ is also the terminal object of $\mx{V}$. The initial object of $\mx{V}$ is denoted by $\varnothing$. By Mac Lane's coherence theorem
\cite[Section~VII.2]{MacLane} (cf.\ also \cite[Theorem~8.4.1.]{YauD}), we may also assume that $(\mx{V},\otimes,\1)$ is a strict monoidal category, i.e., the associator and the unitors are identity morphisms. Unless otherwise stated, all categorical constructions (for instance, coends and Kan extensions) will be carried out in the $\mx{V}$-enriched setting. Such level of generality is needed because $\Stk$ and $\Algst$ are $2$-categories, i.e., they are enriched over the cosmos $(\ms{Cat},\times,\boldsymbol{1})$ of small categories. A~reference for all the tools of enriched category theory that we will use is~\cite{Kel05enr}.

Let $(\ms{E}, \diamond, I)$ be a bicomplete symmetric monoidal closed $\mx{V}$-category and $B$ an object of $\ms E$.
Our guiding example is given by
\[
(\mx{V},\otimes,\1) = (\ms{Cat},\times,\boldsymbol{1}) , \qquad (\ms{E},\diamond,I) = (\Stk, \times_\K, \operatorname{Spec}\K) , \qquad B = \cyinst ,
\]
where $\cyinst$ denotes the stack of cyclotomic gerbes in $\BG$.

In the following $\ms S$ denotes the free $\mx V$-category on the groupoid $\coprod_{n\in \N}\boldsymbol{S}_n$
and ${\ms S}_B \colon {\ms S} \to \ms E$ is the functor that sends $n \mapsto B^n$ and
$\sigma \in \boldsymbol{S}_n$ to the morphism that permutes the factors, where~$B^n$ denotes $B^{\diamond n}$.
\begin{Definition}\label{def S-mod/B}
An $\ms S$-module is a functor $X \colon {\ms S} \to {\ms E}$, that is a family of objects $X_n$ of $\ms E$, $n\in \N$,
endowed with an action of~$\boldsymbol{S}_n$.

A morphism of $\ms S$-modules $X$ and $Y$ is a natural transformation $f\colon X\Rightarrow Y$,
i.e., a family of morphisms $f_n \colon X_n \to Y_n$ compatible with the $\boldsymbol{S}_n$-actions.

For an $\ms S$-module $X$, we will sometimes say that
$X_n$ is the part of $X$ in degree $n$.
The category of $\ms S$-modules will be denoted by $(\ms S - {\rm mod})$.

An $\ms S$-module over $B$ is a pair $(X, e)$, where $X \colon {\ms S} \to {\ms E}$ is an $\ms S$-module
and $e \colon X \Rightarrow {\ms S}_B$ is a natural transformation.
A morphism of $\ms S$-modules over $B$, $(X, e), (Y, \varepsilon)$, is a morphism of $\ms S$-modules $f\colon X \to Y$
such that $\varepsilon \circ f = e$. The category of $\ms S$-modules over $B$ will be denoted by $(\ms S - {\rm mod}/B)$.

When $(\ms{E},\diamond,I) = (\Algst, \times_\K, \operatorname{Spec}\K)$,
for two $\ms S$-modules $X$, $Y$, we say that $X$ is a closed (resp.\ open) sub-$\ms S$-module of $Y$
if $X_n \subset Y_n$ is a full sub-category, the inclusion $X_n \hookrightarrow Y_n$ is a~closed
(resp. open)
immersion, and the $\boldsymbol{S}_n$-action on $Y$ restricts to that on $X$, for every~$n$.
\end{Definition}

The inclusion of $\boldsymbol{S}_n \times \boldsymbol{S}_m$ into $\boldsymbol{S}_{n+m}$ as $\operatorname{Aut}\{1,\dots,n\} \times \operatorname{Aut}\{n+1,\dots,n+m\}$ induces
a~tensor product $+$ for a $\mx{V}$-enriched symmetric monoidal structure on $\ms S$ with unit $0$,
denoted $+\colon \ms S \otimes \ms S \to \ms S$.

Given two $\ms S$-modules $X$ and $Y$, we consider the functor
\begin{align}\label{bifunctor}
(\ms S \otimes \ms S)^{\rm op} \otimes (\ms S \otimes \ms S) &\to (\ms S - {\rm mod}),\\
(m\otimes n)\otimes (k\otimes l) &\mapsto \ms S (m+n, \_) \odot \bigl( X_k \diamond Y_l \bigr) , \nonumber
\end{align}
where $\odot$ denotes the tensor product (also known as copower) in the $\mx{V}$-enriched category $\ms{E}$.
In the next definition we refer to~\cite{MacLane} for the definition and properties of coends.

\begin{Definition}[Day convolution]\label{Day}
For all $\ms S$-modules $X$ and $Y$, the convolution $X \ast Y$ is the functor
${\ms S} \to {\ms E}$ defined by the coend of \eqref{bifunctor}:
\begin{equation*} 
X \ast Y \colon \ k \mapsto \int^{m,n} {\ms S}(m+n, k) \odot \bigl( X_m \diamond Y_n \bigr) .
\end{equation*}
\end{Definition}

\begin{Remark}\label{rmk_conv}\quad
\begin{enumerate}\itemsep=0pt
\item There is a $\mx{V}$-natural isomorphism
\[
X \ast Y \cong \operatorname{Lan}_{+} \bigl( \diamond (X \otimes Y) \bigr) ,
\]
where $\operatorname{Lan}_+$ denotes the left Kan extension along $+ \colon {\ms S}\otimes {\ms S} \to {\ms S}$ \cite[Section~X.4]{MacLane}.

\item For $\ms S$-modules $X$ and $Y$, and for $k\in \N$, the part of $X \ast Y$ in degree $k$ has the following explicit formula
(cf.\ \cite[formula~(2.2)]{Kel05}):
\begin{equation}\label{Kel05}
(X\ast Y)_k = \coprod_{\substack{(m, n)\in \N^2 \\m+n=k}} {\rm Sh}(m, n) \odot X_m \diamond Y_n ,
\end{equation}
where ${\rm Sh}(m, n)$ is the set of $(m,n)$-shuffles in $\boldsymbol{S}_k$, i.e., the permutations
$\sigma \in \boldsymbol{S}_k$ such that $\sigma (1) < \dots < \sigma (m)$ and $\sigma (m+1)<\dots < \sigma (m+n)$.
\end{enumerate}
\end{Remark}

\begin{Proposition}\label{conv/B}
Let $(X, e)$ and $(Y, \varepsilon)$ be $\ms S$-modules over $B$. Then $e$ and $\varepsilon$ induce a natural transformation
$e\diamond \varepsilon \colon X\ast Y \Rightarrow {\ms S}_B$, therefore $(X\ast Y, e\diamond \varepsilon)$ is an $\ms S$-module over $B$.
\end{Proposition}
\begin{proof}
We identify $X\ast Y$ with the left Kan extension $\operatorname{Lan}_{+} \bigl( \diamond (X \otimes Y) \bigr)$ of
$\diamond (X\otimes Y) \colon {\ms S} \otimes {\ms S} \to E$ along $+ \colon {\ms S} \otimes {\ms S} \to {\ms S}$ (Remark~\ref{rmk_conv}\,(1)).
Then, the natural transformations $\operatorname{Lan}_{+} \bigl( \diamond (X \otimes Y) \bigr) \Rightarrow {\ms S}_B$ are in bijection with
the natural transformations $\diamond (X\otimes Y) \Rightarrow {\ms S}_B +$ \cite[Section~X.3\,(9)]{MacLane}.
Since, for every $m,n \in \mathbb{N}$, the morphism $e_m \diamond \varepsilon_n \colon X_m \diamond Y_n \to B^{\diamond m+n}$ is
$(\boldsymbol{S}_m \times \boldsymbol{S}_n)$-equivariant, this yields a natural transformation $e\diamond \varepsilon \colon \diamond (X\otimes Y) \Rightarrow {\ms S}_B +$,
hence we obtain a natural transformation $\operatorname{Lan}_{+} \bigl( \diamond (X \otimes Y) \bigr) \Rightarrow {\ms S}_B$, which we denote with the same symbol
$e\diamond \varepsilon$.
\end{proof}

\begin{Proposition}
For all functors $X, Y, Z \colon {\ms S} \to {\ms E}$,
there is a $\mx{V}$-natural isomorphism $(X \ast Y) \ast Z \simeq X \ast (Y \ast Z)$.
Furthermore, if $(X, e), (Y, \varepsilon), (Z, {\mf e})$ are ${\ms S}$-modules over $B$, then $(e\diamond \varepsilon)\diamond {\mf e} = e\diamond (\varepsilon\diamond {\mf e})$
under the previous isomorphism.
\end{Proposition}
\begin{proof}
The associativity of the convolution is proved in \cite[Proposition~6.2.1]{Loregian} (see \cite[Section~3]{Day70} for the original proof).
To see that $(e\diamond \varepsilon)\diamond {\mf e} = e\diamond (\varepsilon\diamond {\mf e})$, notice that
$(e\diamond \varepsilon ) \diamond {\mf e} = e\diamond (\varepsilon \diamond {\mf e} )$:
${\diamond (\diamond (X\otimes Y) \otimes Z ) \Rightarrow {\ms S}_B ( + (+\otimes {\rm Id}_{\ms S}))}$,
under the identifications $\diamond (\diamond (X\otimes Y) \otimes Z ) = \diamond (X\otimes \diamond (Y \otimes Z ))$
and ${\ms S}_B ( + (+\otimes {\rm Id}_{\ms S})) = {\ms S}_B ( + ( {\rm Id}_{\ms S} \otimes +))$.
Then the claim follows from the properties of the left Kan extension as in the proof of the previous proposition.
\end{proof}

In particular, for a fixed functor $Y \colon {\ms S} \to {\ms E}$, Day convolution determines a $\mx{V}$-functor
${\ms S} \otimes {\ms S}^\op \to \ms{E}$ given by $(m, n) \mapsto Y^{\ast n}(m)$.
This allows us to define the composition of an ${\ms S}$-module over $B$
and a rooted ${\ms S}$-module over $B$.

\begin{Definition}\label{rooted_module}
A rooted ${\ms S}$-module is a pair $(W, \tilde{\varepsilon})$, where $W $ is an ${\ms S}$-module, $\tilde{\varepsilon} \colon W \Rightarrow \Delta_B$ is a natural transformation
and $\Delta_B$ is the constant functor on $B$.
A rooted ${\ms S}$-module over $B$ is a~triple $(W, \varepsilon, \tilde{\varepsilon})$ such that $(W, \varepsilon)$
is an~${\ms S}$-module over $B$ and $(W, \tilde{\varepsilon})$ is a~rooted ${\ms S}$-module.
\end{Definition}

Let $(X, e)$ be an ${\ms S}$-module over $B$, and let $(W, \varepsilon, \tilde{\varepsilon})$ be a rooted ${\ms S}$-module over $B$.
Let us consider the functor
\begin{align}
F \colon \ {\ms S}^{\rm op} \otimes {\ms S} &\to (\ms S - {\rm mod}),\nonumber \\
(n, n') &\mapsto \begin{cases} \varnothing , & {\rm if} \ n \not= n' \\ X_n \times_{B^n} W^{\ast n} (\_) , & {\rm if} \ n=n' , \end{cases}\label{composition_bifunctor}
\end{align}
where the morphism $W^{\ast n}(m) \to B^n$ is the composition of the map induced by the projection
\begin{align*}
W^{\ast n}(m) &= \coprod_{\substack{(m_1, \dots , m_n)\in \N^n \\ m_1 + \dots + m_n =m}}
{\rm Sh}(m_1, \dots , m_n)
\odot W_{m_1} \diamond \dots \diamond W_{m_n} \\
&
\to \coprod_{\substack{(m_1, \dots , m_n)\in \N^n \\ m_1 + \dots + m_n =m}}W_{m_1} \diamond \dots \diamond W_{m_n}
\end{align*}
followed by the morphism induced by $\tilde{\varepsilon}$
\[
 \coprod_{\substack{(m_1, \dots , m_n)\in \N^n \\ m_1 + \dots + m_n =m}}W_{m_1} \diamond \dots \diamond W_{m_n} \to B^n .
\]
Where ${\rm Sh}(m_1, \dots , m_n)$ is the set of $(m_1, \dots , m_n)$-shuffles in $\boldsymbol{S}_m$, that is
the permutations $\sigma \in \boldsymbol{S}_m$ such that
$\sigma (1) < \dots < \sigma (m_1)$, $\sigma (m_1 + 1) <\dots < \sigma (m_1+m_2)$, $\dots$.

\begin{Definition}[composition]\label{composition}
Given an ${\ms S}$-module over $B$, $(X, e)$, and a rooted ${\ms S}$-module over~$B$,
$(W, \varepsilon, \tilde{\varepsilon})$, the composition product (or plethysm)
$X \circ W$ is the $\mx{V}$-functor $\ms{S} \to \ms{E}$ defined by the coend of~\eqref{composition_bifunctor}:
\begin{equation} \label{kellycomp}
X \circ W \colon \ m \mapsto \int^{n} X_n \times_{B^n} W^{\ast n}(m) .
\end{equation}
\end{Definition}

\begin{Proposition}\label{crsmb}
Let $(X, e, \tilde{e})$ and $(W, \varepsilon, \tilde{\varepsilon})$ be rooted ${\ms S}$-modules over $B$.
Then $\varepsilon$ induces, on~$X \circ W$, a structure of ${\ms S}$-module over~$B$, which will be denoted
$\varepsilon' \colon X \circ W \Rightarrow \ms{S}_B$. Furthermore, $\tilde{e}$~induces a structure of rooted ${\ms S}$-module on $X\circ W$.
\end{Proposition}
\begin{proof}
Let $F$ be the functor defined in \eqref{composition_bifunctor}. For any $n\in \mathbb{N}$, $\tilde{e}$ gives a natural transformation
$\alpha_n \colon F(n,n) \Rightarrow \Delta_B$ such that the following diagram commutes for every $\sigma \in \boldsymbol{S}_n$:
\[
\begin{tikzcd}
F(n,n) \arrow{r}{F({\rm id}, \sigma)} \arrow{d}{F( \sigma, {\rm id})} & F(n,n) \arrow{d}{\alpha_n} \\
F(n,n) \arrow{r}{\alpha_n} & \,\Delta_B.
\end{tikzcd}
\]
By the properties of the coend, this gives a natural transformation $X\circ W \Rightarrow \Delta_B$.

The definition of $\varepsilon'$ is similar, it uses the natural transformation $F(n,n) \Rightarrow {\ms S}_B$
induced by $\varepsilon$ as in Proposition~\ref{conv/B}.
\end{proof}

\begin{Definition}\label{DY}
Let $(Y, e)$ be an ${\ms S}$-module over $B$ and let us suppose that
$B$ has an involution~$\iota$. We define~$(DY, \varepsilon, \tilde{\varepsilon})$ to be the
rooted ${\ms S}$-module over $B$
such that $DY_n:= Y_{n+1}$ with the action of $\boldsymbol{S}_n$
seen as the permutations of $\boldsymbol{S}_{n+1}$ that fix $n+1$, $\varepsilon_n \colon DY_n \to B^n$
is given by the first $n$ components of $e_{n+1}$ and
$\tilde{\varepsilon}_n \colon DY_n \to B$ denotes the last component of~$e_{n+1}$
composed with~$\iota$.
\end{Definition}

Let $B$ as in the previous definition with an involution $\iota$.
By $I_1$, we denote the rooted ${\ms S}$-module over $B$ concentrated in degree one
\[
(I_1)_n = \begin{cases}
B & \text{if} \ n=1, \\
\varnothing & \text{otherwise},
\end{cases}
\]
where $\varepsilon_1 = \tilde{\varepsilon}_1 = \operatorname{Id}_B$.

For later use, we define $I_2$ to be the ${\ms S}$-module over $B$ concentrated in degree two
\[
(I_2)_n = \begin{cases}
B & \textit{if} \ n=2, \\
\varnothing & \text{otherwise},
\end{cases}
\]
where $e_2 =(\operatorname{Id}_B, \iota)$ and the symmetric group $\boldsymbol{S}_2$ acts via the involution $\iota$.
Notice that $I_1 = DI_2$.

From now on, we restrict our attention to the case where $(\mx{V},\otimes,\1) = (\ms{Cat},\times,\boldsymbol{1})$
and $(\ms{E},\diamond,I) = (\Stk, \times_\K, \operatorname{Spec}\K)$.
\begin{Proposition} \label{compositionsymmon}
The $\mx{V}$-category of rooted ${\ms S}$-modules over $B$, together with the composition product $\circ$
and the unit $I_1$, forms a monoidal $\mx{V}$-category.
\end{Proposition}

To prove the proposition, we follow Kelly's arguments in \cite[Section~3]{Kel05}. The next Lemma is needed.

\begin{Lemma}\label{lemmacompositionsymmon}
Let $(V, \delta, \tilde \delta), (W, \varepsilon, \tilde \varepsilon)$ be rooted ${\ms S}$-modules over $B$.
Then $(V \circ W)^{*m}$ is naturally isomorphic to $V^{*m} \circ W$.
\end{Lemma}
\begin{proof}
We have
\begin{align*}
(V \circ W)^{*m} &= \int^{n_1, \dots , n_m} {\ms S}(n_1 + \dots + n_m , \_) \times (V \circ W)_{n_1} \times \dots \times (V \circ W)_{n_m} \\
& \cong \int^{n_i, k_i} {\ms S}(n_1 + \dots + n_m , \_) \times_{i=1}^m (V_{k_i} \times_{B^{k_i}} (W^{*k_i})_{n_i})
\\
& \cong \int^{n_i, k_i} (\times_{i=1}^mV_{k_i}
)
\times_{B^{k_1 + \dots + k_m}} {\ms S}(n_1 + \dots + n_m , \_)
\times_{i=1}^m (W^{*k_i})_{n_i} \\
& \cong \int^{k_i} (V_{k_1} \times \dots \times V_{k_m}) \times_{B^{k_1 + \dots + k_m}} (W^{*k_1} * \dots *W^{*k_m}) \\
& \cong \int^{k_i} (V_{k_1} \times \dots \times V_{k_m}) \times_{B^{k_1 + \dots + k_m}} (W^{*k_1+\dots +k_m}) \\
& \cong \int^{k_i} (V_{k_1} \times \dots \times V_{k_m}) \times_{B^{k_1 + \dots + k_m}} \int^t {\ms S}(k_1 + \dots + k_m , t) \times W^{*t} \\
& \cong \int^{k_i, t} {\ms S}(k_1 + \dots + k_m , t) \times (V_{k_1} \times \dots \times V_{k_m}) \times_{B^{k_1 + \dots + k_m}} W^{*t}\\
& \cong \int^{t} (V^{*m})_t \times_{B^{t}} W^{*t},
\end{align*}
where in the second and sixth isomorphisms
we have used the fact that
\[
(\boldsymbol{S}_1 \times_{B^{k_1}} R_1) \times (\boldsymbol{S}_2 \times_{B^{k_2}} R_2) \cong (\boldsymbol{S}_1 \times \boldsymbol{S}_2)\times_{B^{k_1+k_2}}
(R_1 \times R_2) ,
\]
for $R_i$, $\boldsymbol{S}_i$ be $B^{k_i}$-modules;
in the third isomorphism we have used the definition of $*$ and the property that coends commute
with fiber products in the case of stacks;
the fifth isomorphism follows from Yoneda's lemma and the last one follows from the definition of Day convolution.
\end{proof}

\begin{proof}[Proof of Proposition \ref{compositionsymmon}]
Let $(X, e, \tilde{e})$, $(Y, \varepsilon, \tilde{\varepsilon})$, $(Z, {\mf e}, \tilde{\mf e})$ be rooted ${\ms S}$-modules over $B$. Then
\begin{align*}
X\circ (Y\circ Z) &= \int^m X_m \times_{B^m} (Y\circ Z)^{*m} \\
& \cong \int^m X_m \times_{B^m} \bigl(Y^{*m}\circ Z\bigr) \qquad \text{(by Lemma \ref{lemmacompositionsymmon})} \\
& \cong \int^{m, k} X_m \times _{B^m} \bigl(Y^{*m}_k \times_{B^k} Z^{*k}\bigr) \qquad \text{(by definition of composition)} \\
& \cong \int^{k} (X\circ Y)_k \times_{B^k} Z^{*k} \\
& = (X\circ Y) \circ Z .
\end{align*}

As in the proof of Proposition \ref{crsmb}, the structure of rooted ${\ms S}$-module on $X\circ (Y\circ Z)$
is given by the natural transformations
\[
X_n \times_{B^n} (Y\circ Z)^{\ast n} (\_) \Rightarrow \Delta_B \qquad {\rm for} \ n\in \N ,
\]
induced by $\tilde{e}$. On the other hand, the structure of rooted ${\ms S}$-module on $(X\circ Y)\circ Z$
is given by the natural transformations
\[
(X\circ Y)_n \times_{B^n} Z^{\ast n}(\_) \Rightarrow \Delta_B \qquad {\rm for} \ n\in \N ,
\]
induced by the structure of rooted ${\ms S}$-module on $X\circ Y$ (which in turn is induced by $\tilde{e}$).
We see, from the explicit isomorphism $X\circ (Y\circ Z) \cong (X\circ Y)\circ Z$ defined above,
that these two structures of rooted ${\ms S}$-modules coincide.

Concerning the structures of ${\ms S}$-modules over $B$, the one on $(X\circ Y)\circ Z$ is given by the natural transformations
\[
(X\circ Y)_n \times_{B^n} Z^{\ast n}(\_) \Rightarrow {\ms S}_B \qquad {\rm for} \ n\in \N ,
\]
induced by ${\mf e}$, and the one on $X\circ (Y\circ Z)$ is given by the natural transformations
\[
X_n\times_{B^n} (Y\circ Z)^{\ast n}(\_) \Rightarrow {\ms S}_B \qquad {\rm for} \ n\in \N ,
\]
induced by the structure of ${\ms S}$-module over $B$ of $Y\circ Z$ (which in turn is induced by ${\mf e}$).
We see that, under the isomorphism $X\circ (Y\circ Z) \cong (X\circ Y)\circ Z$ defined at the beginning of the proof,
these two structures of rooted ${\ms S}$-modules coincide.

The fact that $I_1$ is a unit for $\circ$ follows from the definition~\eqref{composition_bifunctor} of the functor $F$
whose coend is $X\circ I_1$ and from the following equalities:
\begin{align*}
I_1^{\ast n}(m) &= \coprod_{\substack{(m_1, \dots , m_n)\in \N^n \\ m_1 + \dots + m_n =m}}
{\rm Sh}(m_1, \dots , m_n)
\odot (I_1)_{m_1} \diamond \dots \diamond (I_1)_{m_n} \\
&=
\begin{cases}
\varnothing , & {\rm if} \ n\not= m , \\
 \boldsymbol{S}_n \odot B^n , & {\rm if} \ n= m .
 \end{cases}\tag*{\qed}
\end{align*}\renewcommand{\qed}{}
\end{proof}

\begin{Remark}
For any $Y$, $(-) \circ Y$ has a right adjoint \cite[Section~3]{Kel05}. In general, $X \circ (-)$ does not have a right adjoint,
but it does when $X$ is concentrated in arity one \cite[Lemma~2.9]{DH19}.
\end{Remark}

\begin{Remark}\label{stvsalgst}
If we replace $\Stk$ by $\Algst$ as the target $2$-category, the composition product is no longer well defined since
$\Algst$ is not bi-complete. However, the composition is well defined if $Y_0=\varnothing$.
The reason is that the only colimits involved turn out to be finite coproducts and quotients by finite group actions.
Proposition~\ref{compositionsymmon} and the full faithfulness of the embedding of~$\Algst$ into $\Stk$ then imply that the composition product on
the $2$-category of rooted ${\ms S}$-modules (in~$\Algst$) over~$B$
is $2$-naturally associative (whenever it is defined) with unit~$I_1$.
\end{Remark}

\begin{Example}\label{exSmod}
Let us finally spell out the example of $\ms{S}$-module over $B$ we are interested in,
which serves as a motivation for the theory we are presenting.
Here~$B$ is the stack $\cyinst$ of cyclotomic gerbes in~$\BG$ defined in \cite[Definition~3.3.6]{AGV08}
(see also Section~\ref{recall}),
the involution $\iota\colon \cyinst \to \cyinst$ is induced by the
automorphisms ${\boldsymbol{\mu}}_r \to {\boldsymbol{\mu}}_r$, $\zeta \mapsto \zeta^{-1}$, for $r\geq 1$ (according to \cite[Section~3.5]{AGV08}).

As $n$ varies, the stacks $\bal$, with the evaluation maps
${\rm ev}_n \colon \bal \to B^n$ and the $\boldsymbol{S}_n$-actions given by permutation of the marked points
on $\bal$ (respectively by permutation on~$B^n$), yield an $\ms{S}$-module over $B$ in $\Algst$ (cf.\ Definition~\ref{def S-mod/B}),
which we denote by $(\overline{\mc{B}}, \overline{e})$
and we refer to it as the $\ms{S}$-module of balanced twisted $G$-covers over genus $0$ curves.

 Moreover, let $\balsm$ be the locus in $\bal$ of twisted $G$-covers of smooth $n$-pointed curves of genus $0$
(cf.\ Notation~\ref{balsm}), and let ${\rm ev}_{n|} \colon \balsm \to B^n$ be the restriction of the evaluation map.
The $\boldsymbol{S}_n$-action on $\bal$ restricts to an action on $\balsm$ such that ${\rm ev}_{n|}$ is equivariant.
In this way we obtain an open sub-${\ms S}$-module of $(\overline{\mc{B}}, \overline{e})$ that will be denoted by
$({\mc{B}}, {e})$.
\end{Example}

In order to explain better the statement of our main theorem (Theorem~\ref{mainthm}),
we spell out the details of the previous constructions
in the particular cases that will be relevant for equation~\eqref{maineq}.

First of all, let us consider the rooted $\ms{S}$-module over $B$,
$(I_1 \coprod D\overline{\mc{B}} , \varepsilon , \tilde{\varepsilon})$.
By Definition~\ref{DY}, $(D\overline{\mc{B}})_n = \overline{\mc{B}}_{n+1}$ with $\boldsymbol{S}_n$
acting by permutation only on the first $n$ marked points, ${\varepsilon_n \colon\! (D\overline{\mc{B}})_n \! \to B^n}$
is given by the first $n$ evaluation maps, and
$\tilde{\varepsilon}_n \colon (D\overline{\mc{B}})_n \to B$ is $\iota \circ {\rm ev}_{n+1}^{n+1}$, whereas $\varepsilon_1 = \tilde{\varepsilon}_1 = {\rm Id}_B$. Then, using the concrete expressions~\eqref{Kel05} and~\eqref{composition_bifunctor}
and the definition of coend, we deduce that%
\begin{equation}\label{bcircdb}
\left( \mc{B} \circ (I_1 \coprod D \overline{\mc{B}}) \right)_n =
\coprod_{m} \Bigg( \mc{B}_m \times_{B^m} \coprod_{\substack{\underline{k} \in \N^m \\ |\underline{k}|=n}}
{\rm Sh}(\underline{k}) \times (I_1 \coprod D\overline{\mc{B}})_{\underline{k}} \Bigg) \Big /\boldsymbol{S}_m 	,
\end{equation}
where, for a multi-index $\underline{k} =(k_1, \dots , k_m) \in \N^m$ of length $|\underline{k}|=\sum_{i=1}^m k_i =n$,
$(I_1 \coprod D\overline{\mc{B}})_{\underline{k}} = \times_{i=1}^m (I_1 \coprod D\overline{\mc{B}})_{k_i}$
and the morphisms ${\rm Sh}(\underline{k}) \times (I_1 \coprod D\overline{\mc{B}})_{\underline{k}} \to B^m$
are induced by the $\tilde{\varepsilon}$'s.
Notice that an object in the right-hand side of the above equality is given by the equivalence class
(under the action of $\boldsymbol{S}_m$) of a~stable map in~$\mc{B}_m$, an object in
$(I_1)_{k_i} \coprod \overline{\mc{B}}_{k_i+1}$ for $i=1, \dots , m$ and a shuffle permutation, such that, for any $i=1, \dots , m$,
the evaluation of the stable map in~$\mc{B}_m$ at the $i$-th marked point coincides with the image under $\tilde{\varepsilon}_{k_i}$
of the object in
$(I_1)_{k_i} \coprod \overline{\mc{B}}_{k_i+1}$. Geometrically this means that, if $k_i \geq 2$, we glue together the first stable map with that in~$\overline{\mc{B}}_{k_i+1}$, for $i=1, \dots , m$, and get a balanced stable map. For more details on this, we refer to Section~\ref{section recursive relations}.

Concerning the operation $I_2 \circ D\overline{\mc{B}}$, that we will need later, we have that
\begin{gather*}
(I_2 \circ D\overline{\mc{B}} )_n =
\Bigg( (I_2)_2 \times_{B^2} \coprod_{\substack{\underline{k} \in \N^2 \\ |\underline{k}|=n}}
{\rm Sh}(\underline{k}) \times D\overline{\mc{B}}_{\underline{k}} \Bigg) \Big/\boldsymbol{S}_2 .
\end{gather*}
In particular, its objects are equivalence classes of pairs of stable maps in
$\overline{\mc{B}}_{k_1+1} \times \overline{\mc{B}}_{k_2+1}$
and a~shuffle in ${\rm Sh}(k_1, k_2)$,
such that the evaluations of the two stable maps at the last marked points are related by the involution
$\iota$. This means, geometrically,
that we can glue together the two stable maps along the last marked points to get a balanced map.

The last operation that we will use is $I_2 \times_{B^2} \bigl(D\overline{\mc{B}} \ast D\overline{\mc{B}}\bigr)$,
this is the functor ${\ms S} \to {\ms E}$
which associates to $n$ the following object:
\begin{equation}\label{i2xdb2}
 (I_2)_2 \times_{B^2} \coprod_{\substack{\underline{k} \in \N^2 \\ |\underline{k}|=n}}
{\rm Sh}(\underline{k}) \times D\overline{\mc{B}}_{\underline{k}} .
\end{equation}

\section{Grothendieck groups and their composition structure}\label{section4}

Let us recall the definition of $\Kstk$ from~\cite{Eke09}. Note that there are different versions in the literature (see \cite{BD07, Joy07, Toe09}), all of which are mapped to by $\Kstk$.

\begin{Definition}
The Grothendieck group $\Kstk$ is the abelian group generated by the isomorphism classes $\{\mc{X}\}$
of algebraic $\K$-stacks subject to the relations
\begin{enumerate}\itemsep=0pt
\item[(1)] $\{\mc{X}\} = \{\mc{Y}\} + \{\mc{X} \smallsetminus \mc{Y}\}$ if $\mc{Y}$ is a closed substack of $\mc{X}$, and
\item[(2)] $\{\mc{E}\} = \{\aff^r \times_\K \mc{X}\}$ if $\mc{E}$ is a vector bundle of rank $r$ on $\mc{X}$.
\end{enumerate}
The group $\Kstk$ has, in fact, a commutative ring structure, given on generators by
$\{\mc{X}\} \{\mc{Y}\} = \{ \mc{X} \times_\K \mc{Y} \}$.
\end{Definition}

\begin{Definition}
The Grothendieck group $\Ksnmod$ is the abelian group generated by the isomorphism classes $\{\mc{X}\}$
of algebraic $\K$-stacks with an action of $\boldsymbol{S}_n$ subject to the relations
\begin{enumerate}\itemsep=0pt
\item[(1)] $\{\mc{X}\} = \{\mc{Y}\} + \{\mc{X} \smallsetminus \mc{Y}\}$ if $\mc{Y}$ is a closed
$\boldsymbol{S}_n$-invariant substack of $\mc{X}$ with the restricted action, and
\item[(2)] $\{\mc{E}\} = \{\aff^r \times_\K \mc{X}\}$ if $\pi \colon \mc{E} \to \mc{X}$
is an $\boldsymbol{S}_n$-equivariant vector bundle of rank $r$ on $\mc{X}$, where
the $\boldsymbol{S}_n$-action on $\aff^r \times_\K \mc{X}$ is the extension of the given one on $\mc{X}$
by the trivial action.
\end{enumerate}
\end{Definition}

For our purposes we need to define the Grothendieck group associated to the category of $\ms{S}$-modules over~$B$,
where $(\ms{E},\diamond,I) = (\Algst, \times_\K, \operatorname{Spec}\K)$ (we change the target category
according to Remark~\ref{stvsalgst}).
First of all, let us recall (from Definition~\ref{def S-mod/B}) that a morphism $F\colon (X, e) \to (X', e')$
(respectively an isomorphism) of $\ms{S}$-modules over $B$ is a natural transformation
${F\colon X \Rightarrow X'}$ (respectively a natural isomorphism) such that the diagram
\[
\begin{tikzcd}
X \arrow[dr, Rightarrow, "e"] \arrow[rr, Rightarrow, "F"] & & X' \arrow[dl, Rightarrow, "e'"] \\
 & {\ms S}_B &
\end{tikzcd}
\]
is commutative.
As usual we say that two $\ms{S}$-modules over $B$, $(X, e)$ and $(X', e')$, are isomorphic if there exists an isomorphism
$F \colon (X, e) \to (X', e')$ and we denote with $\{X, e \}$ the isomorphism class of $(X, e)$.
By a vector bundle in the category of $\ms{S}$-modules over $B$ we mean a morphism
$\pi \colon (\mc{E}, e') \to (X,e)$ such that:
for every $n$, $\pi_n \colon \mc{E}_n \to X_n$ is a vector bundle (of some rank $r(n)$) and
it is $\boldsymbol{S}_n$-equivariant.

\begin{Definition}\label{def Ksmod}
The Grothendieck group $\Ksmod$ is the abelian group generated by the isomorphism classes $\{X, e\}$, of
$\ms{S}$-modules over $B$ in $(\Algst, \times_\K, \operatorname{Spec}\K)$,
subject to the relations
\begin{enumerate}\itemsep=0pt
\item[(1)] $\{X, e \} = \{Y, e_| \} + \{(X, e) \smallsetminus (Y, e_|)\}$, if $(Y, e_|)$ is a closed sub-$\ms{S}$-module of $(X, e)$, and
\item[(2)] $\{\mc{E}, e'\} = \{(\aff^{r(n)} \times_\K X_n, e_n)_n\}$, if $\pi \colon (\mc{E}, e') \to (X,e)$ is a vector bundle, where
the $\boldsymbol{S}_n$-action on $\aff^{r(n)} \times_\K X_n$ is the extension of the given one on $X_n$ by the trivial action.
\end{enumerate}
\end{Definition}

The Grothendieck group $\Ksmodr$ associated to the category of rooted $\ms{S}$-modules over~$B$ is defined
similarly and we omit the details.

Notice that $\Ksmod$ and $\KsmodK$ (with the products given by Day convolution)
have natural structures of algebras over $\Kstk$ induced by cartesian product.

\begin{Notation}\label{KsmodK}
In the following, we will consider also $\ms{S}$-modules over $\operatorname{Spec} (\K)$, $X$, where the
natural transformation $X \Rightarrow \ms{S}_{\operatorname{Spec} (\K)}$ is the one induced by the structure morphisms $X_n \to \operatorname{Spec} (\K)$, for all $n$,
and hence we omit it.
The corresponding Grothendieck group $K_0^{\ms{S}}({\Algst}/ \operatorname{Spec} (\K))$ will be denoted $\KsmodK$.

For any $\K$-stack $B$, there is a morphism of $\Kstk$-algebras
\[
\Ksmod \to \KsmodK
\]
that is induced by associating any $\ms{S}$-modules over $B$, $(X,e)$, with $X$.
We will say that two elements $a, b \in \Ksmod$ are \textit{congruent modulo} $B$, in symbol $a\equiv_B b$,
if they have the same image under this morphism.
\end{Notation}

Let $(X, e)$ be an ${\ms S}$-module over $B$, and let $(W, \varepsilon, \tilde{\varepsilon})$ be a rooted ${\ms S}$-module over~$B$
such that $W_0 =\varnothing$.
We observe that the class $\{ X \circ W, \varepsilon' \}$ of the composition product (see Definition~\ref{composition} and Remark~\ref{stvsalgst})
depends only on the classes $\{ X, e \} \in \Ksmod$ and $\{ W, \varepsilon, \tilde{\varepsilon} \} \in \Ksmodr$.
Furthermore, we have the following result, which is part of a theorem announced in \cite{GP06}
in the case of varieties (for a complete proof see \cite[Proposition~1.3.6]{Bagnarol}).

\begin{Proposition}\label{composition_on_Grothendieck_gr}
Let $ \bigl( \Ksmodr \bigr)_1$ be the sub-group generated by the classes $\{ W, \varepsilon, \tilde{\varepsilon} \}$ of rooted
$\ms{S}$-modules over $B$
with $W_0 =\varnothing$. Then the composition product descends to an operation
\[
\circ \colon \ \Ksmod \times \bigl(\Ksmodr \bigr)_1 \to {\Ksmod} ,
\]
which is $\Kstk$-linear in the first argument and such that $\{ X, e \} \circ \{ W, \varepsilon, \tilde{\varepsilon} \} =
\{ X \circ W, \varepsilon' \}$,
for any $(X, e)$, $(W, \varepsilon , \tilde{\varepsilon})$ respectively $\ms{S}$-module and rooted $\ms{S}$-module, over $B$.
\end{Proposition}
\begin{proof}
We follow the steps in \cite[Proposition~1.3.6]{Bagnarol}. By linearity we only need to define
$\{ X, e \} \circ b$ for every $\ms{S}$-modules over $B$, $(X, e)$, and every $b \in \bigl(\Ksmodr \bigr)_1$.

First write $b$ as $\{ Y, \varepsilon, \tilde{\varepsilon} \} - \{ Z, \varepsilon', \tilde{\varepsilon}' \}$
where $(Y, \varepsilon, \tilde{\varepsilon})$, $(Z, \varepsilon', \tilde{\varepsilon}' )$ are rooted
$\ms{S}$-modules over $B$. Moreover, assume that $(Z, \varepsilon', \tilde{\varepsilon}' )$
is a closed sub-$\ms{S}$-modules over $B$ of $(Y, \varepsilon, \tilde{\varepsilon})$.
Let $U:=Y\setminus Z$, and let $\varepsilon_|$, $\tilde{\varepsilon}_|$ the restrictions of
$\varepsilon$, $\tilde{\varepsilon}$ to $U$. Then we have
\begin{align*}
\{ X, e \} \circ \{ Y, \varepsilon, \tilde{\varepsilon} \} = \{ X, e \} \circ ( \{ Z, \varepsilon', \tilde{\varepsilon}' \}
+ \{ U, \varepsilon_|, \tilde{\varepsilon}_| \} )
 =
\{ X, e \} \circ ( \{ (Z, \varepsilon', \tilde{\varepsilon}' ) \amalg
 ( U, \varepsilon_|, \tilde{\varepsilon}_|) \} ) .
\end{align*}
Using the formalism of coend, we have
\[
(X, e) \circ ( (Z, \varepsilon', \tilde{\varepsilon}' ) \amalg
 ( U, \varepsilon_|, \tilde{\varepsilon}_|) ) = \int^{i,j} X_{i+j} \times_{B^{i+j}} Z^{\ast i} \ast U^{\ast j} ,
\]
which is in turn isomorphic to
\[
(X, e) \circ ( U, \varepsilon_|, \tilde{\varepsilon}_|) \amalg
\int^{i \geq 1,j} X_{i+j} \times_{B^{i+j}} Z^{\ast i} \ast U^{\ast j} .
\]
Therefore, in this case, expanding $Z^{\ast i} \ast U^{\ast j}$, $\{ X, e \} \circ b = \{ (X, e), \circ (U, \varepsilon_|, \tilde{\varepsilon}_|)\}$
is equal to
\begin{align*}
\{ (X, e) \circ (Y, \varepsilon, \tilde{\varepsilon} )\} -
\left\{
\int^{i\geq 1 , m, n} \ms{S}(m+n, \_ ) \times \bigl(Z^{\ast i}\bigr)_m \times_{B^i}
\left( \int^{j} X_{i+j} \times_{B^{j}} \bigl((Y\setminus Z)^{\ast j}\bigr)_n \right)
\right\}\! ,
\end{align*}
where $X_{i+j} \to {B^{j}}$ is the composition of $e_{i+j} \colon X_{i+j} \to B^{i+j}$
with the projection $B^{i+j} \to B^j$ to the last $j$-components, while the first $i$ components of $e_{i+j}$
give the map to $B^i$.
Note that, for any $k \in \mathbb{N}$, $(\{ X, e \} \circ b)_k$ depends on
$(Y\setminus Z)^{\ast j}_n$ only for $n<k$. In particular, we have
\begin{align*}
&(\{ X, e \} \circ b)_0= \{ X_0 \}, \\
 &(\{ X, e \} \circ b)_1= \{ (X \circ Y)_1 \} - \{ (X\circ Z)_1 \} .
\end{align*}
In these equations (and from now on in the proof), we have omitted the morphisms to $B$ to simplify the notation,
they are those induced by
$\varepsilon$, $\tilde{\varepsilon}$ and $\varepsilon'$, $\tilde{\varepsilon}'$ from Proposition~\ref{crsmb}.

This suggests a recursive approach for defining $\{ X, e \} \circ b$ when
$b = \{ Y, \varepsilon, \tilde{\varepsilon} \} - \{ Z, \varepsilon', \tilde{\varepsilon}' \}$
and $\{ Z, \varepsilon', \tilde{\varepsilon}' \}$ is not a
 closed sub-$\ms{S}$-modules over~$B$ of $(Y, \varepsilon, \tilde{\varepsilon})$.

Let us consider functors $T \colon \ms{S}\times \ms{S} \to \Algst$ that associate, for any $i, j \in \mathbb{N}$,
a stack $T(i, j)$ together with an $\boldsymbol{S}_i \times \boldsymbol{S}_j$-equivariant morphism $T(i, j) \to B^i \times B^j$.
For any such $T$, let us define
\[
h_Z(T)_{(j,l)} := \left\{ \int^{i\geq 1, m, n}\ms{S}(m+n, l) \times \bigl(Z^{\ast i}\bigr)_m \times_{B^i} T(i+j, n) \right\} .
\]
Using this we define $\{ X, e \} \circ_k b \in K_0^{\boldsymbol{S}_k}({\Algst})$ inductively as follows:
\begin{align*}
&\{ X, e \} \circ_0 b:= \{ X_0 \} , \\
&\{ X, e \} \circ_k b:= \{ (X \circ Y)_k \} - h_Z (X \circ_{k-1} b)_{(0, k)} .
\end{align*}
Finally, we define $\{ X, e \} \circ b := \sum_{k\geq 0} \{ X, e \} \circ_k b$.
\end{proof}

\section{Stratification}\label{section5}

In this section, we introduce a stratification of $\mathcal B^{\rm bal}_{0,n}(G)$
according to the combinatorial type of its objects. This will be achieved using the notion of \textit{gerby tree} associated to
$\mathcal B G$.

Let us first recall the following definitions.

\begin{Definition}[\cite{BM96}] A \textit{graph} $\tau$ is a triple $(F_\tau, j_\tau, R_\tau)$, where
\begin{itemize}\itemsep=0pt
\item $F_\tau$ is a finite set, called the set of flags of $\tau$,
\item $j_\tau \colon F_\tau \to F_\tau$ is an involution, and
\item $R_\tau \subset F_\tau \times F_\tau$ is an equivalence relation.
\end{itemize}
\end{Definition}
Associated to every graph $\tau = (F_\tau, j_\tau, R_\tau)$, there is a triple of sets $(V_\tau, E_\tau, L_\tau)$,
which are defined as follows:
\begin{itemize}\itemsep=0pt
\item $V_\tau$ is the quotient of $F_\tau$ by $R_\tau$, it is called the set of vertices of $\tau$;
\item $E_\tau$ is the set of orbits of $j_\tau$ with two elements, it is called the set of edges of $\tau$;
\item $L_\tau$ is the set of fixed points of $j_\tau$, it is called the set of leaves of $\tau$.
\end{itemize}
Finally, we associate to every vertex $v\in V_\tau$
\begin{enumerate}\itemsep=0pt
\item[(1)] the subset $F_\tau (v) \subset F_\tau$ of flags whose equivalence class under $R_\tau$ is $v$, and
\item[(2)] the valence $n(v):= \# F_\tau (v)$ of $v$.
\end{enumerate}

\begin{Definition}\label{geometric realization}
The {\it geometric realization} $|\tau|$ of a graph $\tau$ is the topological space constructed as follows.
Let us consider the disjoint union of intervals and singletons
\[
T = \bigsqcup_{i\in L_\tau} [0_i,{1/2}_i) \sqcup \bigsqcup_{i\in F_\tau \setminus L_\tau} [0_i,{1/2}_i] \sqcup \bigsqcup_{v\in V_\tau} \{ v\} .
\]
For each $v\in V_\tau$, identify the elements of the set $\{ 0_i \mid i\in F_\tau (v)\}$ with $v$.
For each edge ${\{ i_1, i_2 \} \in E_\tau}$, identify ${1/2}_{i_1}$ with ${1/2}_{i_2}$. In this way, one obtains a set $|\tau|$ together with a~surjection
$T\to |\tau|$. Then the topology on~$|\tau|$ is the quotient topology.
\end{Definition}

\begin{Definition}
A {\it tree} is a graph $\tau$ such that its geometric realization is simply connected,
i.e., $|\tau|$ is connected and $\# V_\tau = \# E_\tau +1$.
An $N$-tree is a tree $\tau$ together with a bijection $l_\tau \colon L_\tau \to N$.
A ($N$-)tree $\tau$ is \textit{stable} if $n(v) \geq 3$, for any $v\in V_\tau$.
As usual, when $N=\{ 1, \dots , n\}$ we will speak of $n$-trees instead of $N$-trees.
\end{Definition}

Let $\bar{\mathcal I}_{\boldsymbol{\mu}} (\mathcal B G)$ be the stack of cyclotomic gerbes in $\mathcal B G$,
let $\bar{I}_{\boldsymbol{\mu}} (\mathcal B G)$ be its coarse moduli space (cf.\ Section~\ref{recall_sub}), and let
$\iota\colon \cyinst \to \cyinst$ be the involution induced by the
automorphisms ${\boldsymbol{\mu}}_r \to {\boldsymbol{\mu}}_r$, $\zeta \mapsto \zeta^{-1}$, for $r\geq 1$ (cf.\ Example~\ref{exSmod}).
In order to simplify the notation, we will use the same symbol
$\iota\colon \bar{I}_{\boldsymbol{\mu}} (\mathcal B G) \to \bar{I}_{\boldsymbol{\mu}} (\mathcal B G)$
for the morphism induced by $\iota$ on the coarse moduli space.
The following definition is similar to \cite[Definition~2.14]{AJT16}.

\begin{Definition}\label{gerbytree}
A {\it gerby tree} $\tilde \tau$ associated to $\mathcal B G$ is given by a tree $\tau =(F_\tau, j_\tau, R_\tau)$
together with a map $\mf{m} \colon F_\tau \to \bar{I}_{\boldsymbol{\mu}} (\mathcal B G)$ such that $\mf{m}(i) = \iota \circ \mf{m}(i')$,
whenever the flags $i, i' \in F_\tau$ form an edge $\{ i, i'\} \in E_\tau$.
We will denote with $\mf{o} \colon F_\tau \to \mathbb Z$ the function that associates to any flag $i\in F_\tau$
the integer $r\in \mathbb Z$ such that $\mf{m}(i ) \in \bar{I}_{{\boldsymbol{\mu}}_r} (\mathcal B G)$.
A gerby $N$-tree is a gerby tree $\tilde \tau$ together with a~bijection $l_\tau \colon L_\tau \to N$.
A gerby ($N$-)tree $\tilde \tau$ is \textit{stable}, if the underlying ($N$-)tree is stable.
\end{Definition}

In the following, with a slight abuse of notation, for a gebry tree~$\tilde \tau$,
we denote either with $V_{\tilde \tau}$ (resp.\ $E_{\tilde \tau}$, $L_{\tilde \tau}$)
or with $V_\tau$ (resp.\ $E_\tau$, $L_\tau$) the set of vertices (resp.\ edges, leaves) of the underlying tree $\tau$.
Furthermore, we denote with $E'_\tau$ (resp.\ $E'_{\tilde \tau}$) the set $F_\tau \setminus L_\tau$.

\begin{Definition}
An \textit{isomorphism of graphs} $\varphi \colon \tau \to \sigma$ is a bijective map $\varphi_F \colon F_\tau \to F_\sigma$
such that $\varphi_F \circ j_\tau = j_\sigma \circ \varphi_F$ and $(\varphi_F \times \varphi_F)(R_\tau) = R_\sigma$.
In particular $\varphi_F$ induces a bijection $\varphi_V \colon V_\tau \to V_\sigma$, which commutes with the projections,
and a bijection $\varphi_L \colon L_\tau \to L_\sigma$.

An \textit{isomorphism of $N$-trees} $\varphi \colon (\tau, l_\tau) \to (\sigma, l_\sigma)$ is an isomorphism of the underlying trees
$\varphi \colon \tau \to \sigma$ such that $l_\tau = l_\sigma \circ \varphi_L$.
The automorphism group of a tree $\tau$ (resp.\ an $N$-tree $(\tau, l_\tau)$) will be denoted ${\rm Aut}(\tau)$
(resp.\ ${\rm Aut}(\tau, l_\tau)$).

An \textit{isomorphism of gerby trees} (resp. \textit{$N$-trees}) $\tilde \varphi \colon \tilde \tau \to \tilde \sigma$ is an isomorphism of the underlying trees
(resp.\ $N$-trees) $\varphi \colon \tau \to \sigma$ such that $\mf{m}_\sigma \circ \varphi_F =\mf{m}_\tau$.
The automorphism group of a gerby ($N$-)tree $\tilde \tau$ will be denoted ${\rm Aut}(\tilde \tau)$.

In this article, we choose a representative
for any equivalence class of stable gerby $n$-trees associated to $\mathcal B G$
and we denote with $\tilde{\Gamma}_{0,n}$ the set of these representatives.
\end{Definition}

We recall the following result, for which we provide a proof (following~\cite{GP06}) for lack of a~suitable reference.
\begin{Lemma}
The set $\tilde{\Gamma}_{0,n}$ of isomorphism classes of stable gerby $n$-trees associated to $\mathcal B G$ is finite.
\end{Lemma}
\begin{proof}
Let $\tilde{\tau} \in \tilde{\Gamma}_{0,n}$.
By definition, $|\tau|$ is connected and
\[
\# V_\tau = \# E_\tau +1 .
\]
The right-hand side of the previous equation is equal to $\frac{1}{2}\left( \sum_{v\in V_\tau} n(v) -n \right) +1$,
hence
\[
n-2= \sum_{v\in V_\tau} \left( n(v)-2 \right) .
\]
The stability condition ($n(v)\geq 3$, for every $v\in V_\tau$) implies that $\# V_\tau \leq n-2$,
therefore
\[
\# F_\tau = \sum_{v\in V_\tau} n(v)= n-2 +2\# V_\tau \leq 3(n-2) .
\]
The claim follows since the isomorphism classes of gerby $n$-trees associated to $\mathcal B G$, with a fixed number of flags, is finite.
 \end{proof}

Given a gerby tree $\tilde \tau$, in the literature there are notions of $\tilde \tau$-marked pre-stable curves and
of $\tilde \tau$-marked twisted stable maps to $\mathcal B G$ (see, e.g.,
\cite[Definitions~2.18 and~2.20]{AJT16}). In this article, we give a different definition of
$\tilde \tau$-marked twisted stable maps to $\mathcal B G$, which is more convenient for our purposes.
To this aim, for any $v\in V_{\tilde \tau}$ let
$\boldsymbol{\mf{m}}(v):= (\mf{m}(i))_{i\in F_{\tilde \tau}(v)} \in \bar{I}_{\boldsymbol{\mu}} (\mathcal B G)^{F_{\tilde \tau}(v)}$.
Then consider the open and closed sub-stack $\balmv$ of $\mc{B}^\mr{bal}_{0,F_{\tilde\tau}(v)}(G)$
(Definition~\ref{balc}).
The evaluation maps ${\rm ev}^i_{F_{\tilde\tau}(v)}$, for $i\in E'_{\tilde \tau} \cap F_{\tilde\tau}(v)$, yield a morphism
\[
{\rm ev}_{\tilde \tau , E} \colon \ \prod_{v\in V_{\tilde \tau}} \balmv \to \cyinst^{E'_{\tilde \tau}} .
\]
Furthermore, let us define
\[
\cyinst^{(\tilde \tau)} := \bigl\{ (\mathcal G_i, \phi_i)_{i\in E'_{\tilde \tau}} \in \cyinst^{E'_{\tilde \tau}} \mid
(\mathcal G_i, \phi_i) = \iota( \mathcal G_{i'}, \phi_{i'}) , \, \forall \{ i, i' \} \in E_{\tilde \tau} \bigr\} .
\]
\begin{Definition}\label{tau_tilde_marked_cover}
Let $\tilde \tau$ be a gerby tree associated to $\mathcal B G$.
The \textit{stack of $\tilde \tau$-marked twisted $G$-covers of genus $0$}, which will be denoted $\mathcal B_{\tilde \tau}(G)$,
is defined as the fiber product
\begin{equation*}\label{tau_tilde_marked_cover_fiber_product}
\cyinst^{(\tilde \tau)} \times_{\cyinst^{E'_{\tilde \tau}}} \prod_{v\in V_{\tilde \tau}} \balmv ,
\end{equation*}
with respect to the inclusion $\cyinst^{(\tilde \tau)} \to \cyinst^{E'_{\tilde \tau}}$
and the morphism ${\rm ev}_{\tilde \tau , E}$.

Furthermore, we define $\mathcal B_{\tilde \tau}^{\rm sm}(G)$ to be the fiber product
\[
\cyinst^{(\tilde \tau)} \times_{\cyinst^{E'_{\tilde \tau}}} \prod_{v\in V_{\tilde \tau}} \balmvsm ,
\]
where $\balmvsm \subset \balmv$ is the open sub-stack consisting of twisted $G$-covers
of smooth curves (see Notation \ref{balsm}).
\end{Definition}

\begin{Remark}
By definition, an object of $\mathcal B_{\tilde \tau}(G)$ over a scheme $T$ is a triple
\[
\bigl( (\mathcal G_i, \phi_i)_{i\in E'_{\tilde \tau}}, (\mathcal{C}_v, (\Sigma_k)_{k\in F_{\tilde \tau}(v)}, f_v)_{v\in V_{\tilde \tau}} ,
 (\Phi_i, \rho_i)_{i\in E'_{\tilde\tau}} \bigr) ,
\]
where $(\mathcal G_i, \phi_i)_{i\in E'_{\tilde \tau}} \in \cyinst^{(\tilde \tau)}(T)$,
$(\mathcal{C}_v, (\Sigma_k)_{k\in F_{\tilde \tau}(v)}, f_v) \in \balmv (T)$, and
 $(\Phi_i, \rho_i) \colon (\mathcal G_i, \phi_i) \! \to (\Sigma_i, {f_v}_{|\Sigma_i})$ is an isomorphism, for every
$v \in V_{\tilde\tau}$, $i\in F_{\tilde\tau}(v) \cap E'_{\tilde \tau}$.

If $\{ i, i'\} \in E_{\tilde \tau}$, since $(\mathcal G_i, \phi_i) = \iota( \mathcal G_{i'}, \phi_{i'})$
(by definition of $\cyinst^{(\tilde \tau)}$), we get an isomorphism
\[
(\Phi_{i'}, \rho_{i'}) \circ (\Phi_i, \rho_i)^{-1} \colon \ (\Sigma_i, {f_v}_{|\Sigma_i}) \to \iota (\Sigma_{i'}, {f_{v'}}_{|\Sigma_{i'}}) ,
\]
where $v'\in V_{\tilde \tau}$ is such that $i' \in F_{\tilde \tau}(v')$.
\end{Remark}

It follows from this remark that, for any stable gerby $n$-tree associated to $\mathcal{B}G$, $\tilde \tau$, we can apply
the gluing construction of \cite[Section~5 and Appendix~A]{AGV08} that yields a morphism
\begin{equation}\label{xi_tildetau}
\xi_{\tilde \tau} \colon \ \mathcal B_{\tilde \tau}(G) \to \mathcal B^{\rm bal}_{0,n}(G) .
\end{equation}

\begin{Proposition}\label{xitautilde_proper}
The morphism $\xi_{\tilde \tau}$ in \eqref{xi_tildetau} is representable and proper.
\end{Proposition}
\begin{proof}
To simplify the notation, in this proof we denote with $e_i$ the evaluation morphism ${\rm ev}^i_N$,
whenever the set $N$ is clear from the context.

Let us choose an ordering of the edges of $\tilde \tau$,
$E_{\tilde \tau} = \{ \{ i_1, i_1' \}, \dots , \{ i_k, i_k' \} \}$. Hence
$F_{\tilde \tau} \setminus L_{\tilde \tau}= \{ i_1, i_1', \dots , i_k, i_k' \}$ and $j_{\tilde \tau}(i_\ell) = i_{\ell}'$, for $\ell = 1, \dots , k$.
Notice that, since $\tau$ is a tree, $\# V_{\tilde \tau} = k+1$ and, for any $\ell$, $[i_\ell] \not= [i_\ell']$.
In particular there is an $\ell$ such that either the vertex $[i_\ell]$ or $[i_\ell']$ belongs only to the edge $\{i_\ell, i_\ell'\}$
and we suppose, without loss of generality, that $[i_k']$ satisfies this condition.

Let us proceed by induction on $k\geq 1$ (if $k=0$, $\xi_{\tilde \tau}$ is the identity).
For $k=1$, we have that
\[
{\mathcal B}_{\tilde \tau}(G) \cong {\mathcal B}^{\rm bal}_{0,F_{\tilde \tau}([i_1])}(G) \tensor[_{\check e_{i_1}}]{\times}{_{e_{i_1}'}}
{\mathcal B}^{\rm bal}_{0,F_{\tilde \tau}([i_1'])}(G) ,
\]
where $\check e_{i_1} := \iota \circ e_{i_1}$, and $\xi_{\tilde \tau}$ is the morphism in \cite[Proposition 5.2.1\,(1)]{AGV08}.
Therefore, $\xi_{\tilde \tau}$ is representable by the aforementioned proposition.
Finally, $\xi_{\tilde \tau}$ is proper since the stacks $\mathcal B^{\rm bal}_{0,n}(G)$
are proper \cite[Theorem~1.4.1]{AV02}, in particular, they are separated and also $\mathcal B_{\tilde \tau}(G)$
is proper.

Let us suppose that the result is true for gerby trees with $k-1$ edges and
let $\tilde \tau$ be a gerby $n$-tree with $\# E_{\tilde \tau} = k$.
Let ${\tilde \tau}_1$ be the gerby tree with
\begin{itemize}\itemsep=0pt
\item $F_{{\tilde \tau}_1}:= [i_1] \cup \dots \cup [i_{k}] \subset F_{\tilde \tau}$;
\item $j_{{\tilde \tau}_1}$ equal to $j_{\tilde \tau}$ on $F_{{\tilde \tau}_1} \setminus \{ i_{k} \}$
and $j_{{\tilde \tau}_1}(i_{k})=i_{k}$;
\item $R_{{\tilde \tau}_1} = R_{\tilde \tau} \cap (F_{{\tilde \tau}_1} \times F_{{\tilde \tau}_1})$;
\item $\mf{m}_{{\tilde \tau}_1} ={\mf{m}_{\tilde \tau}}_{|F_{{\tilde \tau}_1}}$.
\end{itemize}
Furthermore, let $\tilde \tau_2$ be the gerby tree with
\begin{itemize}\itemsep=0pt\samepage
\item $F_{{\tilde \tau}_2}:= [i_k'] \subset F_{\tilde \tau}$;
\item $j_{{\tilde \tau}_2}$ equal to the identity;
\item $R_{{\tilde \tau}_2} = F_{{\tilde \tau}_2} \times F_{{\tilde \tau}_2}$;
\item $\mf{m}_{{\tilde \tau}_2} ={\mf{m}_{\tilde \tau}}_{|F_{{\tilde \tau}_2}}$.
\end{itemize}
Since $\tilde \tau_1$ has $k-1$ edges, $\xi_{\tilde \tau_1}\colon \mathcal B_{\tilde \tau_1}(G) \to \mathcal B^{\rm bal}_{0,m}(G)$
is representable and proper, $m:= \# L_{\tilde \tau_1}$. Moreover,
\[
\mathcal B_{\tilde \tau}(G) \cong \mathcal B_{\tilde \tau_1}(G) \tensor[_{\check e_{i_k}}]{\times}{_{e_{i_k'}}}
\mathcal B_{\tilde \tau_2}(G) ,
\]
where $\check e_{i_k} := \iota \circ e_{i_k}$, and under this identification
$\xi_{\tilde \tau}$ is the composition of the following morphisms:%
\[
\mathcal B_{\tilde \tau_1}(G) \tensor[_{\check e_{i_k}}]{\times}{_{e_{i_k'}}} \mathcal B_{\tilde \tau_2}(G)
\stackrel{(\xi_{\tilde \tau_1}, {\rm Id})}{\longrightarrow}
\mathcal B_{0, m}^{\rm bal}(G) \tensor[_{\check e_{i_k}}]{\times}{_{e_{i_k'}}} \mathcal B_{\tilde \tau_2}(G)
\to \mathcal B_{0, n}^{\rm bal}(G) ,
\]
where the morphism to the right is the gluing one of \cite[Proposition 5.2.1\,(1)]{AGV08}.
\end{proof}

Furthermore, we have the following result.
\begin{Proposition}\label{xitautilde_immersion}
The morphism $\xi_{\tilde \tau}$ in \eqref{xi_tildetau} is a closed immersion.
\end{Proposition}
To prove the previous proposition, we need a preliminary result.
\begin{Lemma}\label{auttau_trivial}
The automorphism group of a stable $n$-tree $(\tau, l_\tau)$ is trivial.
\end{Lemma}
\begin{proof}
We proceed by induction on the number $\# E_\tau$ of edges. If $\# E_\tau =0$ the claim is obvious.
So, let us assume that $\#E_\tau =k+1$ and that the result holds true for trees with $k$ edges.
Let us fix a numbering of the edges, $E_\tau = \{ \{i_1,i_1' \}, \dots, \{i_{k+1},i_{k+1}' \} \}$
as in the proof of Proposition~\ref{xitautilde_proper}, that is such that
$j_{\tau}(i_\ell) = i_{\ell}'$, for $\ell = 1, \dots , k+1$, and $[i_{k+1}']$ belongs only to the edge $\{i_{k+1}, i_{k+1}'\}$.
Notice that, since $\tau$ is stable, $F_\tau([i_{k+1}']) \cap L_\tau \not= \varnothing$.
Let now $\varphi \in {\rm Aut}(\tau, l_\tau)$. By definition, $\varphi (f)=f$ for any $f\in F_\tau([i_{k+1}']) \cap L_\tau$,
hence $\varphi ([i_{k+1}']) = [i_{k+1}']$ and $\varphi (i_{k+1}')=i_{k+1}'$
(in other words $\varphi$ is the identity on $[i_{k+1}']$).
Let now $\tau_1$ be the tree defined as in the proof of Proposition~\ref{xitautilde_proper},
with $F_{\tau_1}=F_\tau \setminus F_\tau ([i_{k+1}'])$, and let $l_{\tau_1}$ be a $n_1$-marking, where
$n_1:= \#L_{\tau_1}$. Notice that $(\tau_1, l_{\tau_1})$ is a stable $n_1$-tree
and that the restriction of $\varphi$ to $F_{\tau_1}$ is an automorphism of $(\tau_1, l_{\tau_1})$.
By induction we conclude that $\varphi = {\rm Id}$.
\end{proof}

\begin{proof}[Proof of Proposition \ref{xitautilde_immersion}]
The claim follows essentially by the same proofs of Proposition~10.11 and Corollary~10.22
in Chapter~XII of~\cite{ACG11}. The main ingredients in these proofs are: the existence of Kuranishi families
(see \cite[Chapter~XVI, Section~9]{ACG11}), and the fact that the automorphism group of any stable gerby $n$-tree
$(\tilde \tau, l_\tau)$ is trivial (Lemma \ref{auttau_trivial}).
\end{proof}

The image of $\mathcal B_{\tilde \tau}^{\rm sm}(G)$ under $\xi_{\tilde \tau}$ is the locus of (balanced) twisted $G$-covers of combinatorial type $\tilde \tau$,
which we denote $\mathcal B^{\rm bal}(G, \tilde \tau)$.
Then we have the following corollary of Proposition \ref{xitautilde_immersion}.
\begin{Corollary}\label{cor_xitautilde_immersion}
The image $\xi_{\tilde \tau} ( \mathcal B_{\tilde \tau}(G))$ is closed in $\mathcal B^{\rm bal}_{0,n}(G)$
and $\mathcal B^{\rm bal}(G, \tilde \tau)$ is open in $\xi_{\tilde \tau} ( \mathcal B_{\tilde \tau}(G))$.
\end{Corollary}
\begin{proof}
Notice that $\mathcal B^{\rm sm}_{\tilde \tau}(G)$ is open in $\mathcal B_{\tilde \tau}(G)$
by definition,
and $\mathcal B^{\rm bal}(G, \tilde \tau)$ is open in the image of $\xi_{\tilde \tau}$ because it is
the complement of a closed locus, by Proposition \ref{xitautilde_proper}.
\end{proof}

\section{Recursive relations}\label{section recursive relations}
In this section, we prove our main result, Theorem \ref{mainthm}, which yields recursive formulae that express each class
$\{\bal\} \in \Kstk$ in terms of the classes of the open loci in $\mathcal{B}^\mathrm{bal}_{0,m}(G)$, for $m\leq n$,
corresponding to stable maps with smooth domain. To this aim we need some preliminary definitions and results.
First of all, following \cite[Definition 4.2]{BagnarolMN}, let us give the following definition.
\begin{Definition}
Let $\mc{X}$ be an algebraic $\K$-stack. A locally closed decomposition of $\mc{X}$ is a~morphism
$f \colon \mc{Y} \to \mc{X}$ of algebraic $\K$-stacks that satisfies the following conditions:
\begin{itemize}\itemsep=0pt
\item the restriction of $f$ to each connected component of $\mc{Y}$ is a locally closed immersion;
\item $f$ is bijective on closed points.
\end{itemize}
\end{Definition}

Let us recall that a locally closed immersion is a morphism that can be factored as $j\circ i$, where $i$ is a closed immersion and $j$ is an open immersion.

In the following, we will use the next result, which is well known, but we provide an elementary proof for completeness.
\begin{Lemma}
Let $f \colon \mc{Y} \to \mc{X}$ be a locally closed decomposition. Let $\mc{Y}_1, \dots , \mc{Y}_m$
be the connected components of $\mc{Y}$. Then the following equality holds true in~$\Kstk$:
\[
\{\mc{X}\} = \sum_{k=1}^m \{f( \mc{Y}_k)\} .
\]
\end{Lemma}
\begin{proof}
By definition, $f( \mc{Y}_1), \dots , f( \mc{Y}_m)$ are locally closed substacks of $\mc{X}$ such that
\[
\mc{X} = \bigsqcup_{k=1}^m f( \mc{Y}_k) .
\]
We proceed by induction on $m$. Let us assume $m>1$, since the case $m=1$ is trivial.
Now, for $\overline{f( \mc{Y}_m)}$ being the closure of $f( \mc{Y}_m)$, using the properties of $\Kstk$ and the fact that
$f( \mc{Y}_m)$ is open in $\overline{f( \mc{Y}_m)}$, we have the following equalities:
\begin{align*}
\{ \mc{X}\} &= \{ \overline{f( \mc{Y}_m)} \} + \{ \mc{X} \setminus \overline{f( \mc{Y}_m)}\} \\
&= \{ {f( \mc{Y}_m)} \} + \{ \overline{f( \mc{Y}_m)} \setminus f( \mc{Y}_m) \} + \{ \mc{X} \setminus \overline{f( \mc{Y}_m)}\} .
\end{align*}
Now we observe that
\[
\overline{f( \mc{Y}_m)} \setminus f( \mc{Y}_m) = \bigsqcup_{k=1}^{m-1} \bigl( \overline{f( \mc{Y}_m)} \setminus f( \mc{Y}_m) \bigr) \cap f( \mc{Y}_k)
= \bigsqcup_{k=1}^{m-1}\overline{f( \mc{Y}_m)} \cap f( \mc{Y}_k)
\]
and
\[
\mc{X} \setminus \overline{f( \mc{Y}_m)} = \bigsqcup_{k=1}^{m-1} \bigl( \mc{X} \setminus \overline{f( \mc{Y}_m)} \bigr) \cap f( \mc{Y}_k) ,
\]
which are locally closed decompositions of the corresponding right-hand sides. By induction, we deduce that
\begin{align*}
\{ \mc{X}\} &= \{ {f( \mc{Y}_m)} \} + \sum_{k=1}^{m-1} \bigl\{ \overline{f( \mc{Y}_m)} \cap f( \mc{Y}_k) \bigr\} +
\sum_{k=1}^{m-1} \bigl\{ \bigl( \mc{X} \setminus \overline{f( \mc{Y}_m)} \bigr) \cap f( \mc{Y}_k) \bigr\} \\
&= \{ {f( \mc{Y}_m)} \} + \sum_{k=1}^{m-1} \{ f( \mc{Y}_k) \} .\tag*{\qed}
\end{align*}\renewcommand{\qed}{}
\end{proof}

\begin{Proposition}\label{mainprop}
For any integer $n\geq 3$, let $\tilde{\Gamma}_{0,n}$ be a set of representatives of the isomorphism classes
of stable gerby $n$-trees associated to $\mathcal B G$.
Let $(\overline{\mc{B}}, \overline{e})$ and
$({\mc{B}}, {e})$ be the $\ms S$-modules of Example~{\rm \ref{exSmod}}.
Then the following equalities hold true in $\Ksnmod$:
\begin{gather}
 \sum_{\tilde{\tau} \in \tilde{\Gamma}_{0,n}} \bigl\{ V_{\tilde{\tau}} \times \mc{B}^{\rm sm}_{\tilde{\tau}} (G)\bigr\} =
\bigl\{\mc{B} \circ \bigl( I_1 \amalg D\overline{\mc{B}} \bigr)\bigr\}_n , \label{vertice marcato} \\
 \sum_{\tilde{\tau} \in \tilde{\Gamma}_{0,n}} \bigl\{ E_{\tilde{\tau}} \times \mc{B}^{\rm sm}_{\tilde{\tau}} (G)\bigr\} =
\bigl\{ I_2 \circ D \overline{\mc{B}} \bigr\}_n , \label{spigolo marcato} \\
 \sum_{\tilde{\tau} \in \tilde{\Gamma}_{0,n}} \bigl\{ (F_{\tilde{\tau}} \setminus L_{\tilde{\tau}}) \times
\mc{B}^{\rm sm}_{\tilde{\tau}} (G)\bigr\} = \bigl\{(I_2)_2 \times_{B^2} (D\overline{\mc{B}} \ast D\overline{\mc{B}} )_n \bigr\} .\label{bandiera marcata}
\end{gather}
\end{Proposition}
\begin{proof}
To prove \eqref{vertice marcato}, we show that there is a locally closed decomposition
\begin{equation}\label{vertice marcato'}
\coprod_{\tilde{\tau} \in \tilde{\Gamma}_{0,n}}V_{\tilde{\tau}} \times \mc{B}^{\rm sm}_{\tilde{\tau}} (G)
\to
\coprod_{m} \Biggl( \mc{B}_m \times_{B^m} \coprod_{\substack{\underline{k} \in \N^m \\ |\underline{k}|=n}}
{\rm Sh}(\underline{k}) \times \bigl(I_1 \amalg D\overline{\mc{B}}\bigr)_{\underline{k}} \Biggr) \Big/\boldsymbol{S}_m ,
\end{equation}
where
\[
\bigl(I_1 \amalg D\overline{\mc{B}}\bigr)_{\underline{k}}= \bigl(I_1 \amalg D\overline{\mc{B}}\bigr)_{k_1}\times \dots \times
\bigl(I_1 \amalg D\overline{\mc{B}}\bigr)_{k_m} ,
\]
so that the claim then follows since (by \eqref{kellycomp} in Definition \ref{composition}) the class in $\Kstk$ of the target of \eqref{vertice marcato'}
is the right-hand side of \eqref{vertice marcato}.
Notice that for any $n$ the set $\tilde{\Gamma}_{0,n}$ is finite and there are only non-zero contributions for $3\leq m \leq n$ in \eqref{vertice marcato'}.

Let $\tilde{\tau} \in \tilde{\Gamma}_{0,n}$, let $w\in V_{\tilde{\tau}}$ and let $m = |F_{\tilde{\tau}}(w)|$.
We first construct a morphism
\begin{equation}\label{vertice marcato''}
\Phi_w \colon \ \boldsymbol{S}_m \times \{w\} \times \mc{B}^{\rm sm}_{\tilde{\tau}} (G)
\to
 \mc{B}_m \times_{B^m} \coprod_{\substack{\underline{k} \in \N^m \\ |\underline{k}|=n}}
{\rm Sh}(\underline{k}) \times \bigl(I_1 \amalg D\overline{\mc{B}}\bigr)_{\underline{k}} ,
\end{equation}
where, with a slight abuse of notation, the $\boldsymbol{S}_m$ in the left-hand side denotes the set of bijections
$\{ 1, \dots , m \} \to F_{\tilde{\tau}}(w)$. To this aim, let $\theta \in \boldsymbol{S}_m$ be such a bijection
and let $f_i := \theta (i)$, for $i=1, \dots, m$.
Composing the natural morphism $\mc{B}^{\rm sm}_{\tilde\tau}(G) \to \prod_{v\in V_{\tilde \tau}} \balmvsm$
(see Definition~\ref{tau_tilde_marked_cover}) with the projection to the factor indexed by~$w$,
we obtain a morphism
$\mc{B}^{\rm sm}_{\tilde\tau}(G) \to \mc{B}_{0,F_{\tilde\tau}(w)}^{\rm sm}(G)$, which gives, via $\theta$,
the first component $\mc{B}^{\rm sm}_{\tilde\tau}(G) \to \mc{B}_m$ of \eqref{vertice marcato''}.

The second component
of \eqref{vertice marcato''} is defined as follows.
Let $\tilde{\tau}_w$ be the graph $(F_{\tilde{\tau}} \setminus F_{\tilde{\tau}}(w), j_{\tilde{\tau}_w}, R_{\tilde{\tau}_w})$, where
\[
j_{\tilde{\tau}_w} (f):=
\begin{cases}
j_{\tilde{\tau}} (f) & {\rm if} \ j_{\tilde{\tau}} (f) \not\in F_{\tilde{\tau}}(w) , \\
f & {\rm otherwise} ,
\end{cases}
\]
and $R_{\tilde{\tau}_w}$ is the restriction of $R_{\tilde{\tau}}$ to $F_{\tilde{\tau}} \setminus F_{\tilde{\tau}}(w)$.
Notice that the geometric realization $|\tilde{\tau}_w|$ of $\tilde{\tau}_w$
is the complement in $|\tilde{\tau}|$ of
\[
\bigcup_{f\in L_{\tilde{\tau}} \cap F_{\tilde{\tau}}(w)} [0_f,{1/2}_f) \bigcup
\bigcup_{f\in F_{\tilde{\tau}}(w) \setminus L_{\tilde{\tau}}}[0_f, {1/2}_f] ,
\]
in particular, the connected components of $|\tilde{\tau}_w|$ are in bijection with the flags
in $F_{\tilde{\tau}(w)} \setminus L_{\tilde{\tau}}$.
For any $i=1, \dots, m$ such that $f_i \in F_{\tilde{\tau}(w)} \setminus L_{\tilde{\tau}}$, let $\tilde{\tau}_w^i$ the sub-graph of $\tilde{\tau}_w$
that contains $j_{\tilde{\tau}}(f_i)$ and whose geometric realization is a connected component of $|\tilde{\tau}_w|$.
For any $i=1,\dots, m$,
let
\[
k_i:=
\begin{cases}
1 & {\rm if} \ f_i \in L_{\tilde{\tau}} , \\
|L_{\tilde{\tau}^i_w}| -1 & {\rm otherwise} ,
\end{cases}
\]
so that $k_1 +\dots +k_m=n$.
The marking $l_{\tilde{\tau}} \colon L_{\tilde{\tau}} \leftrightarrow \{ 1, \dots , n \}$
(recall that $\tilde{\tau}$ is an $n$-tree) induces a unique
marking of the leaves of $\tilde{\tau}_w^i$ such that
$j_{\tilde{\tau}}(f_i)$ is marked with $k_i+1$ and the other leaves are marked from $1$ to $k_i$
preserving the order determined by $l_{\tilde{\tau}}$.
This data determines a $(k_1, \dots, k_m)$-shuffle $\sigma$ in the following way:
setting $k_0=0$, for any $i=1, \dots , m$,
if $k_i>1$, $\sigma (k_0+ \dots +k_{i-1}+1)< \sigma (k_0+ \dots +k_{i-1}+2) < \dots < \sigma (k_0+ \dots +k_{i-1}+k_i)$
are the indices of the first $k_i$ leaves of $\tilde{\tau}^i_w$ (i.e., those different from $j_{\tilde{\tau}}(f_i)$)
with respect to the marking $l_{\tilde{\tau}}$; if $k_i=1$, $\sigma (k_0+ \dots +k_{i-1}+1)= l_{\tilde{\tau}}(f_i)$.

Finally, for any $i=1, \dots , m$ such that $k_i>1$, the subgraph $\tilde{\tau}_w^i$ is stable,
hence we have a morphism
$\mc{B}^{\rm sm}_{\tilde\tau}(G) \to \prod_{v \in V_{\tilde{\tau}_w^i}} \mc{B}_{0,F_{\tilde\tau}(v)}^{\rm sm}(G)$,
which is the composition of the
 natural morphism
$\mc{B}_{\tilde\tau}^{\rm sm}(G) \to \prod_{v \in V_{\tilde\tau}} \mc{B}_{0,F_{\tilde\tau}(v)}^{\rm sm}(G)$
(see Definition~\ref{tau_tilde_marked_cover})
 with the projection to the product of factors corresponding to the vertices
$v \in V_{\tilde{\tau}_w^i}$.
Then the gluing construction of \cite[Section~5 and Appendix~A]{AGV08}, together with the enumeration of the leaves of
$\tilde{\tau}_w^i$ as described before, yields a~morphisms
$\mc{B}^{\rm sm}_{\tilde\tau}(G) \to (D\overline{\mc{B}})_{k_i}$, where the indices $i$ are such that $k_i>1$.
In the remaining cases where $k_i=1$
the composition of
the first component $\mc{B}^{\rm sm}_{\tilde\tau}(G) \to \mc{B}_m$ in \eqref{vertice marcato''} with the evaluation
${\rm ev}_i \colon \mc{B}_m \to B$ yields a map $\mc{B}^{\rm sm}_{\tilde\tau}(G) \to (I_1)_1$.
In this way we define the third component of $\Phi_w$ in~\eqref{vertice marcato''}.

By construction $\Phi_w$ is $\boldsymbol{S}_m$-equivariant,
hence it induces a morphism
\[
\{ w\} \times \mc{B}^{\rm sm}_{\tilde{\tau}} (G)
\to
\Biggl( \mc{B}_m \times_{B^m} \coprod_{\substack{\underline{k} \in \N^m \\ |\underline{k}|=n}}
{\rm Sh}(\underline{k}) \times \Bigl(I_1 \coprod D\overline{\mc{B}}\Bigr)_{\underline{k}} \Biggr) \Big/\boldsymbol{S}_m ,
\]
which canonically induces the morphism \eqref{vertice marcato'}.

In order to prove that \eqref{vertice marcato'} is a locally closed decomposition, it suffices to show that
it has the following two properties:
it is bijective on closed points; \eqref{vertice marcato''} is a locally closed immersion
(because the quotient morphism by a finite group is open and closed).
 We first prove the second property. To this aim let $\theta \in \boldsymbol{S}_m$ and let
\[
\Phi_{w, \theta} \colon \ \{ \theta \} \times \{ w \} \times \mc{B}^{\rm sm}_{\tilde{\tau}} (G) \to
\mc{B}_m \times_{B^m} \Big(I_1 \coprod D\overline{\mc{B}}\Big)_{\underline{k}}
\]
be the restriction of $\Phi_w$, where we have adopted the same notation as before, so that
$\underline{k}$ refers to the shuffle determined by $\theta$.
In the next discussion $w$ and $\theta$ are fixed, so we omit them.
Furthermore, we view $\mc{B}^{\rm sm}_{\tilde{\tau}} (G) \subset
 \mc{B}_m \times_{B^m} \prod_{i : k_i>1} \mc{B}_{\tilde{\tau}^i_w}^{\rm sm}(G)$ as an open substack
(in Definition~\ref{tau_tilde_marked_cover}, we permute the factors to bring $\balmwsm$ in the first position to the left,
and we embed $\balmwsm \subset \mc{B}_m$).
Then the morphism $\Phi_{w, \theta}$ is the composition of this open inclusion
followed~by
\begin{gather*}
\mr{pr}_{\mc{B}_m} \times \prod_{i\colon k_i = 1} \mr{ev}_i \circ \mr{pr}_{\mc{B}_m} \times
\prod_{i\colon k_i > 1} \xi_{\tilde{\tau}^i_w} \circ \mr{pr}_{\mc{B}_{\tilde{\tau}^i_w}(G)} \colon\\
\qquad{}
\mc{B}_m \times_{B^m} \prod_{i\colon k_i>1} \mc{B}_{\tilde{\tau}^i_w}(G) \to
\mc{B}_m \times_{B^m} \prod_{i\colon k_i=1}B \times \prod_{i\colon k_i>1} D\overline{\mc{B}}_{k_i} ,
\end{gather*}
where $\mr{pr}$ denotes the projection to the corresponding factor.
The claim now follows from Proposition~\ref{xitautilde_immersion}.

We now prove that \eqref{vertice marcato'} is bijective on closed points.
Notice that, since \eqref{vertice marcato''} is a locally closed immersion, \eqref{vertice marcato'} is injective on closed points.
To see that \eqref{vertice marcato'} is surjective on closed points, notice that any object in the target can be obtained
as the glueing of twisted $G$-covers over smooth basis.

The proof of \eqref{spigolo marcato} and \eqref{bandiera marcata} are similar to that of \eqref{vertice marcato}, so we sketch
them.
The right-hand side of \eqref{spigolo marcato} is the class of
\[
\Biggl( (I_2)_2 \times_{B^2} \coprod_{\substack{\underline{k} \in \N^2 \\ |\underline{k}|=n}}
{\rm Sh}(\underline{k}) \times D\overline{\mc{B}}_{\underline{k}} \Biggr) \Big/\boldsymbol{S}_2 ,
\]
while the right-hand side of \eqref{bandiera marcata} is the class of
\[
 (I_2)_2 \times_{B^2} \coprod_{\substack{\underline{k} \in \N^2 \\ |\underline{k}|=n}}
{\rm Sh}(\underline{k}) \times D\overline{\mc{B}}_{\underline{k}} ,
\]
where we do not pass to the quotient by $\boldsymbol{S}_2$.

The choice of a flag $f\in F_{\tilde{\tau}} \setminus L_{\tilde{\tau}}$ corresponds
to the oriented edge $\overrightarrow{\mathfrak{e}} =(f, j_{\tilde{\tau}}(f))$
of $\tilde{\tau}$. For any such oriented edge $\overrightarrow{\mathfrak{e}}$,
if we cut $|\tilde{\tau}|$ at the point corresponding to ${1/2}_f$, we obtain two
subtrees~$\tilde{\tau}^1$ and~$\tilde{\tau}^2$ of $\tilde{\tau}$, where $f \in F_{\tilde{\tau}^1}$
and $j_{\tilde{\tau}}(f) \in F_{\tilde{\tau}^2}$. Let $k_1 +1$ (resp.\ $k_2 +1$)
be the number of leaves of~$\tilde{\tau}^1$ (resp.\ $\tilde{\tau}^2$), then $k_1 + k_2 = n$.
As before this decomposition of $\tilde{\tau}$ determines a shuffle in ${\rm Sh}(\underline{k})$
and gluing the various irreducible components of the objects in $\mc{B}^{\rm sm}_{\tilde{\tau}} (G)$
according to $\tilde{\tau}^1$ and $\tilde{\tau}^2$ yields a~morphism
\begin{equation}\label{bandiera marcata'}
\{ \overrightarrow{\mathfrak{e}} \} \times \mc{B}^{\rm sm}_{\tilde{\tau}} (G) \to
(I_2)_2 \times_{B^2} \coprod_{\substack{\underline{k} \in \N^2 \\ |\underline{k}|=n}}
{\rm Sh}(\underline{k}) \times D\overline{\mc{B}}_{\underline{k}} .
\end{equation}
In this way, we obtain a locally closed decomposition
\[
\coprod_{\tilde{\tau} \in \tilde{\Gamma}_{0,n}}(F_{\tilde{\tau}}\setminus L_{\tilde{\tau}}) \times \mc{B}^{\rm sm}_{\tilde{\tau}} (G)
\to
 (I_2)_2 \times_{B^2} \coprod_{\substack{\underline{k} \in \N^2 \\ |\underline{k}|=n}}
{\rm Sh}(\underline{k}) \times D\overline{\mc{B}}_{\underline{k}}
\]
that proves \eqref{bandiera marcata}.
At the same time, exchanging the orientation of $\overrightarrow{\mathfrak{e}}$ corresponds to the
natural $\boldsymbol{S}_2$-action on the target of \eqref{bandiera marcata'}, so we also have a locally closed decomposition
\[
\coprod_{\tilde{\tau} \in \tilde{\Gamma}_{0,n}}E_{\tilde{\tau}} \times \mc{B}^{\rm sm}_{\tilde{\tau}} (G)
\to
\Biggl( (I_2)_2 \times_{B^2} \coprod_{\substack{\underline{k} \in \N^2 \\ |\underline{k}|=n}}
{\rm Sh}(\underline{k}) \times D\overline{\mc{B}}_{\underline{k}} \Biggr) \Big/\boldsymbol{S}_2
\]
that proves \eqref{spigolo marcato}.
\end{proof}

\begin{Theorem}\label{mainthm}
Let $\bigl(\overline{\mc{B}}, \overline{e}\bigr)$ be the ${\ms S}$-module defined in Example~{\rm \ref{exSmod}}
and let $({\mc{B}}, e)$ be the ${\ms S}$-sub-module given by the open locus corresponding to the stable maps
with smooth domain $($see Example~{\rm \ref{exSmod})}. Then the following relation holds true:
\begin{align}\label{maineq}
& \bigl\{\overline{\mc{B}}\bigr\} \,\equiv_B \, \{\mc{B}\} \circ \bigl( \{I_1\} + D\bigl\{\overline{\mc{B}}\bigr\} \bigr) + \{I_2\} \circ D\bigl\{\overline{\mc{B}}\bigr\} -
\bigl\{ (I_2)_2 \times_{B^2} \bigl(D\overline{\mc{B}} \ast D\overline{\mc{B}}\bigr) \bigr\},
\end{align}
where $B$ is $\cyinst$ and both sides of the relation belong to {$\Ksmod$}
$($see Notation~{\rm \ref{KsmodK})}.
\end{Theorem}

\begin{proof}
Since the relation \eqref{maineq} is an equality in $\KsmodK$, it suffices to prove that $\bigl\{\overline{\mc{B}}_n\bigr\} \in \Ksnmod$
coincides with the
part in degree $n$ of the right-hand side, for any $n\in \N$.
To this aim, let $\tilde{\Gamma}_{0,n}$ be a set of representatives of the isomorphism classes
of stable gerby $n$-trees associated to~$\mathcal B G$.
Since $\overline{\mc{B}}_n$ is the disjoint union of the locally closed substacks $\mc{B}^{\rm bal}(G, \tilde{\tau})$
(Corollary~\ref{cor_xitautilde_immersion}), for $\tilde{\tau} \in \tilde{\Gamma}_{0,n}$, we obtain the following equalities in $\Kstk$:
\begin{equation}\label{strata}
\bigl\{ \overline{\mc{B}}_n \bigr\} = \sum_{\tilde{\tau} \in \tilde{\Gamma}_{0,n}} \bigl\{ \mc{B}^{\rm bal}(G, \tilde{\tau}) \bigr\}
 = \sum_{\tilde{\tau} \in \tilde{\Gamma}_{0,n}} \bigl\{ \mc{B}^{\rm sm}_{\tilde{\tau}} (G)\bigr\} ,
\end{equation}
where, in the second equality, we have used the fact that $\mc{B}^{\rm bal}(G, \tilde{\tau})$ and $\mc{B}^{\rm sm}_{\tilde{\tau}} (G)$ are isomorphic via
the morphism $\xi_{\tilde{\tau}}$ defined in the previous section.

For any $\tilde{\tau} \in \tilde{\Gamma}_{0,n}$, there is a canonical injection
\begin{equation}\label{GP_can_inj}
F_{\tilde{\tau}} \hookrightarrow V_{\tilde{\tau}} \sqcup E_{\tilde{\tau}} \sqcup L_{\tilde{\tau}}
\end{equation}
that maps $F_{\tilde{\tau}}$ into the complement of a (canonical) point
$\ast \in V_{\tilde{\tau}} \sqcup E_{\tilde{\tau}} \sqcup L_{\tilde{\tau}}$ \cite[formula~(4.2)]{GP06}.
On the one hand, \eqref{strata} yields
\[
\sum_{\tilde{\tau} \in \tilde{\Gamma}_{0,n}} \{ (\ast \, \sqcup \, F_{\tilde{\tau}}) \times \mc{B}^{\rm sm}_{\tilde{\tau}} (G)\}
= \{ \overline{\mc{B}}_n \} + \sum_{\tilde{\tau} \in \tilde{\Gamma}_{0,n}} \{ F_{\tilde{\tau}} \times \mc{B}^{\rm sm}_{\tilde{\tau}} (G)\} ,
\]
on the other hand, identifying
$\ast \sqcup F_{\tilde{\tau}}$ with $V_{\tilde{\tau}} \sqcup E_{\tilde{\tau}} \sqcup L_{\tilde{\tau}}$
via \eqref{GP_can_inj}, we obtain the next equality:
\begin{gather*}
\bigl\{ \overline{\mc{B}}_n \bigr\} + \sum_{\tilde{\tau} \in \tilde{\Gamma}_{0,n}} \bigl\{ F_{\tilde{\tau}} \times \mc{B}^{\rm sm}_{\tilde{\tau}} (G)\bigr\} \\
\qquad {}=
\sum_{\tilde{\tau} \in \tilde{\Gamma}_{0,n}} \bigl\{ V_{\tilde{\tau}} \times \mc{B}^{\rm sm}_{\tilde{\tau}} (G)\bigr\} +
 \sum_{\tilde{\tau} \in \tilde{\Gamma}_{0,n}} \bigl\{ E_{\tilde{\tau}} \times \mc{B}^{\rm sm}_{\tilde{\tau}} (G)\bigr\} + \sum_{\tilde{\tau} \in \tilde{\Gamma}_{0,n}} \bigl\{ L_{\tilde{\tau}} \times \mc{B}^{\rm sm}_{\tilde{\tau}} (G)\bigr\} .
\end{gather*}
As $F_{\tilde{\tau}} = (F_{\tilde{\tau}} \setminus L_{\tilde{\tau}}) \sqcup L_{\tilde{\tau}}$, the previous equality reduces to
\begin{gather*}
\bigl\{ \overline{\mc{B}}_n \bigr\} +
\sum_{\tilde{\tau} \in \tilde{\Gamma}_{0,n}} \bigl\{ (F_{\tilde{\tau}}\setminus L_{\tilde{\tau}}) \times \mc{B}^{\rm sm}_{\tilde{\tau}} (G)\bigr\} =
\sum_{\tilde{\tau} \in \tilde{\Gamma}_{0,n}} \bigl\{ V_{\tilde{\tau}} \times \mc{B}^{\rm sm}_{\tilde{\tau}} (G)\bigr\} + \sum_{\tilde{\tau} \in \tilde{\Gamma}_{0,n}} \bigl\{ E_{\tilde{\tau}} \times \mc{B}^{\rm sm}_{\tilde{\tau}} (G)\bigr\} .
\end{gather*}
The theorem now follows directly from Proposition \ref{mainprop}.
\end{proof}

\begin{Remark}\label{remmainthm}
For any $n\geq 3$, the part of the right-hand side of \eqref{maineq} in degree $n$ expresses $\bigl\{\overline{\mc{B}}_n\bigr\} \in \Ksnmod$
in terms of $I_1, I_2, \overline{\mc{B}}_3, \dots , \overline{\mc{B}}_{n-1}, {\mc{B}}_n$, considered as stacks over $B$ (as we explain below).
Therefore, by induction, the class $\bigl\{\overline{\mc{B}}_n\bigr\}$ can be computed knowing
$I_1, I_2, {\mc{B}}_3, \dots , \allowbreak {\mc{B}}_{n-1}, {\mc{B}}_n$, since $\overline{\mc{B}}_3 = {\mc{B}}_3$.

This follows from the expressions below for $\bigl( \{\mc{B}\} \circ \bigl(\{I_1\}+ D\bigl\{\overline{\mc{B}}\bigr\}\bigr)\bigr)_n$,
$\bigl(\{I_2\} \circ D\bigl\{\overline{\mc{B}}\bigr\}\bigr)_n$ and $\bigl(\bigl\{ (I_2)_2 \times_{B^2} \bigl(D\overline{\mc{B}} \ast D\overline{\mc{B}}\bigr) \bigr\}\bigr)_n$.
For the first one, using \eqref{bcircdb}, we have
\begin{align*}
\bigl( \{\mc{B}\} \circ \bigl(\{ I_1\}+ D\bigl\{\overline{\mc{B}}\bigr\} \bigr)\bigr)_n &= \bigl\{ \mc{B} \circ \bigl(I_1 \amalg D \overline{\mc{B}}\bigr)_n \bigr\} \nonumber \\
&=\sum_m \Biggl\{ \Biggl( \mc{B}_m \times_{B^m} \coprod_{\substack{\underline{k} \in \N^m \\ |\underline{k}|=n}}
{\rm Sh}(\underline{k}) \times \bigl(I_1 \amalg D\overline{\mc{B}}\bigr)_{\underline{k}} \Biggr) \Big/\boldsymbol{S}_m \Biggr\} 	 .
\end{align*}
Notice that $m\geq 3$ (otherwise $\mc{B}_m = \varnothing$) and,
for the multi-index $\underline{k}=(k_1, \dots , k_m)$, $k_1, \dots , k_m \geq 1$
(otherwise $\bigl(I_1 \coprod D\overline{\mc{B}}\bigr)_{k_i} = \varnothing$).
Now, the condition $|\underline{k}|= \sum_{i=1}^m k_i=n$ implies, on one hand that
\[
m \leq k_1+\dots +k_m=n , \qquad \mbox{hence} \quad m\leq n ,
\]
on the other hand that, for every $j=1, \dots , m$,
\[
k_j = n - \sum_{i\not= j}k_i \leq n-2 .
\]
For the second one,
\begin{gather}\label{main_eq_2nd+}
\bigl(\{I_2\} \circ D\bigl\{\overline{\mc{B}}\bigr\}\bigr)_n = \Biggl\{
\Biggl( (I_2)_2 \times_{B^2} \coprod_{\substack{\underline{k} \in \N^2 \\ |\underline{k}|=n}}
{\rm Sh}(\underline{k}) \times D\overline{\mc{B}}_{\underline{k}} \Biggr) \Big/\boldsymbol{S}_2 \Biggr\} .
\end{gather}
Since $\bigl(D\overline{\mc{B}}\bigr)_k = \varnothing$ for $k<2$,
we deduce as before that $\{I_2\} \circ D\bigl\{\overline{\mc{B}}\bigr\}_n$ is determined by
$\{ I_2 \}$ and by $\bigl\{\overline{\mc{B}}_k\bigr\}$ for $k=3, \dots, n-1$.
The same arguments apply to the summand $\bigl(\bigl\{ (I_2)_2 \times_{B^2} \bigl(D\overline{\mc{B}} \ast D\overline{\mc{B}}\bigr) \bigr\}\bigr)_n$
(see~\eqref{i2xdb2}), hence we omit the details.
\end{Remark}

\section[Computations in the Grothendieck group of mixed Hodge structures]{Computations in the Grothendieck group\\ of mixed Hodge structures}\label{section7}

In this section, we work over the field of complex numbers $\mathbb{C}$
and derive formulae for the class, in the Grothendieck group, of the mixed Hodge structure over $\mathbb{Q}$
of $\overline{\mc{B}}_n$. Since $\overline{\mc{B}}_n$ and ${\mc{B}}_n$ are Deligne--Mumford stacks
and we consider their cohomology with rational coefficients,
we can replace them with their coarse moduli spaces (cf.~\cite{Beh04}),
 which are respectively projective and quasi-projective varieties,
and will be denoted with the same symbols, $\overline{\mc{B}}_n$ and ${\mc{B}}_n$, in this section.

Following \cite{GP06} (see also \cite{BagnarolMN}), let $K_0^{\boldsymbol{S}_n}({\rm Var})$
be the Grothendieck group of complex $\boldsymbol{S}_n$-varieties. It is the quotient of the
free abelian group on the isomorphism classes
$[X]$ of quasi-projective varieties with an action of $\boldsymbol{S}_n$ by the subgroup generated by the relations
$[X] - [Y] - [X\setminus Y]$ whenever $Y$ is a closed $\boldsymbol{S}_n$-invariant subvariety of $X$.

The notions of ${\ms S}$-module, ${\ms S}$-module over $B$, and rooted ${\ms S}$-module makes sense
for varieties. Notice that in this case an ${\ms S}$-module is called an
${\ms S}$-variety in \cite{GP06} and \cite{BagnarolMN}.
The Grothendieck groups $K_0^{\ms S}({\rm Var})$, $K_0^{\ms S}({\rm Var}/B)$
and $K_0^{{\ms S}, r}({\rm Var}/B)$ are defined accordingly.
Moreover, we can consider the Day convolution and the composition in this case,
as a consequence of Proposition \ref{composition_on_Grothendieck_gr},
we have an operation
\[
\circ \colon \ K_0^{\ms S}({\rm Var}/B) \times \bigl(K_0^{{\ms S}, r}({\rm Var}/B)\bigr)_1 \to K_0^{\ms S}({\rm Var}) ,
\]
which is $K_0^{\ms S}({\rm Var}/B)$-linear in the first argument and such that $[X]\circ [Y] = [X\circ Y]$
for all ${\ms S}$-varieties over $B$, $X, Y$, with $Y$ rooted.

From the results of the previous section we have the following relation
\begin{align}\label{maineq_var}
 \bigl[\overline{\mc{B}}\bigr] \, &\equiv_B \, [\mc{B}] \circ \bigl( [I_1] + D\bigl[\overline{\mc{B}}\bigr] \bigr) +
[I_2] \circ D\bigl[\overline{\mc{B}}\bigr] -
\bigl[ (I_2)_2 \times_{B^2} \bigl(D\overline{\mc{B}} \ast D\overline{\mc{B}}\bigr) \bigr]
\end{align}
in $K_0^{\ms S}({\rm Var})$.
For computations, we need to express the three summands in the right-hand side of this equation in a more explicit way.
The part in degree $n$ of the first summand is
\begin{align}
\bigl( [\mc{B}] \circ \bigl([ I_1 ]+ D \bigl[\overline{\mc{B}}\bigr] \bigr)\bigr)_n &= \Bigl[ \mc{B} \circ \Bigl(I_1 \coprod D \overline{\mc{B}}\Bigr)_n \Bigr] \nonumber \\
&=\sum_{m=3}^n \Biggl[ \Biggl( \mc{B}_m \times_{B^m} \coprod_{\substack{\underline{k} \in \N^m \\ |\underline{k}|=n}}
{\rm Sh}(\underline{k}) \times \Bigl(I_1 \coprod D\overline{\mc{B}}\Bigr)_{\underline{k}} \Biggr) \Big/\boldsymbol{S}_m \Biggr] 	.\label{main_eq__var_1st+}
\end{align}
For any $m=3, \dots , n$ and for any sequence
$\boldsymbol{c}=(c_1, \dots , c_m) \in {\overline{I}_{\boldsymbol{\mu}}(\BG)}^m$
of conjugacy classes of~$G$, let $\mc{B}_{\boldsymbol{c}} = \balc \cap \mc{B}_m$
(cf.\ Definition~\ref{balc}), let us denote $\balc$ with $\overline{\mc{B}}_{\boldsymbol{c}}$, and let us stress that in this section
we denote with the same symbols stacks and their coarse moduli spaces.
Then, using this notation, we rewrite the right-hand side of~\eqref{main_eq__var_1st+} as follows:
\begin{gather}\label{main_eq__var_1st+_2}
 \sum_{m=3}^n \Biggl[ \Biggl( \coprod_{\boldsymbol{c} \in {\overline{I}_{\boldsymbol{\mu}}(\BG)}^m}
\coprod_{\substack{\underline{k} \in \N^m \\ |\underline{k}|=n}} \mc{B}_{\boldsymbol{c}} \times
{\rm Sh}(\underline{k}) \times_{i=1}^m \Bigl( (\{ c_i \})_{k_i} \coprod \overline{\mc{B}}_{k_i, \iota (c_i)}\Bigr) \Biggr) \Big/\boldsymbol{S}_m \Biggr] 	,
\end{gather}
where $\overline{\mc{B}}_{k_i,c_i} \subseteq \overline{\mc{B}}_{k_i +1}$ consists of covers such that the evaluation
at the last marked point is $c_i$, and $(\{ c_i \})_{k_i} = \operatorname{Spec} (\mathbb{C})$ if $k_i=1$, empty otherwise.
Using the same notation, for the second summand of \eqref{maineq_var}, we rewrite \eqref{main_eq_2nd+} as follows:
\begin{align}
\bigl([I_2] \circ D[\overline{\mc{B}}]\bigr)_n &= \Biggl[
\Biggl( (I_2)_2 \times_{B^2} \coprod_{\substack{\underline{k} \in \N^2 \\ |\underline{k}|=n}}
{\rm Sh}(\underline{k}) \times D\overline{\mc{B}}_{\underline{k}} \Biggr) \Big/\boldsymbol{S}_2 \Biggr] \nonumber \\
&=
\Biggl[ \Biggl( \coprod_{c \in {\overline{I}_{\boldsymbol{\mu}}(\BG)}} \coprod_{\substack{\underline{k} \in \N^2 \\ |\underline{k}|=n}}
{\rm Sh}(\underline{k}) \times \overline{\mc{B}}_{k_1,c} \times \overline{\mc{B}}_{k_2, \iota (c)} \Biggr) \Big/\boldsymbol{S}_2 \Biggr].\label{main_eq_var_2nd+}
\end{align}
Finally, the third summand is
\begin{gather}\label{main_eq_var_3rd+}
\bigl(\bigl[ (I_2)_2 \times_{B^2} \bigl(D\overline{\mc{B}} \ast D\overline{\mc{B}}\bigr) \bigr]\bigr)_n =
\sum_{c \in {\overline{I}_{\boldsymbol{\mu}}(\BG)}} \sum_{\substack{\underline{k} \in \N^2 \\ |\underline{k}|=n}}
\bigl[ {\rm Sh}(\underline{k}) \times \overline{\mc{B}}_{k_1,c} \times \overline{\mc{B}}_{k_2, \iota (c)} \bigr] .
\end{gather}

Let us now consider the category MHS of mixed Hodge structures over $\mathbb{Q}$,
for which our basic references are \cite{PS08} and~\cite{Vois02}
(it might be useful to consult also \cite[Section~5]{GP06} and \cite[Section~2.2]{BagnarolMN} for the purposes of the present work).
Since MHS is a $\mathbb{Q}$-linear rigid
abelian tensor category, we can consider the associated Grothendieck group $K_0({\rm MHS})$, which is isomorphic to the
Grothendieck group of pure Hodge structures over $\mathbb{Q}$ (see, e.g., \cite[Corollary~3.9]{PS08}).
Moreover, for any integer $n\geq 0$, we denote with $K_0^{\boldsymbol{S}_n}({\rm MHS})$ (respectively, with $K_0^{\ms S}({\rm MHS})$)
 the Grothendieck group
of the category of functors from $\boldsymbol{S}_n$ (respectively, from ${\ms S}$) to MHS.
The Hodge--Grothendieck characteristic (called the Serre characteristic in~\cite{GP06},
and denoted HG-characteristic hereafter)
yields a group homomorphism
\[
\mathfrak{e} \colon \ K_0^{\boldsymbol{S}_n}({\rm Var}) \to K_0^{\boldsymbol{S}_n}({\rm MHS}) ,
\]
for any $n$, where $\mathfrak{e}$ is defined on generators by
\[
\mathfrak{e}([X]):= \sum_{i\geq 0}(-1)^i \bigl[ \bigl( H^i_{\rm c} (X, \mathbb{Q}), W_\bullet , F^\bullet \bigr) \bigr] ,
\]
here $\left( H^i_{\rm c} (X, \mathbb{Q}), W_\bullet , F^\bullet \right)$ is the natural mixed Hodge structure
of the compactly supported cohomology group $H^i_{\rm c} (X, \mathbb{Q})$ of the variety~$X$, as defined in~\cite{DeligneII},
\cite{DeligneIII}.
Actually, $K_0^{\ms S}({\rm Var})$ and $K_0^{\ms S}({\rm MHS})$ have a richer structure, namely that of complete composition
algebra \cite[Theorems~2.2 and~5.1]{GP06}, and
$\mathfrak{e} \colon K_0^{\ms S}({\rm Var}) \to K_0^{\ms S}({\rm MHS})$ is a morphism of complete filtered algebras with composition operations \cite[Section~5]{GP06}.
In particular, $\mathfrak{e} \colon K_0({\rm Var}) \to K_0({\rm MHS})$ is a ring homomorphism
(where the ring structures are induced by the cartesian product on $K_0({\rm Var})$ and by the tensor product
on $K_0({\rm MHS})$). However, notice that the composition operation of~\cite{GP06} is different from ours,
so we can not apply their results directly.

For completeness we recall that, for every mixed Hodge structure $(V, W_\bullet, F^\bullet)$
over $\mathbb{Q}$,
its Hodge numbers are defined as
\[
h^{p,q}(V, W_\bullet, F^\bullet) := \dim_{\mathbb{C}} {\rm gr}^p_F {\rm gr}^W_{p+q}(V\otimes \mathbb{C}) ,
\]
and its Hodge--Euler polynomial is
\[
{\rm \bf e} (V, W_\bullet, F^\bullet) := \sum_{p, q \in \mathbb{Z}} h^{p,q}(V, W_\bullet, F^\bullet) u^p v^q \in
\mathbb{Z}\bigl[u, v, u^{-1}, v^{-1}\bigr] .
\]
This induces a ring homomorphism
\[
{\rm \bf e} \colon \ K_0({\rm MHS}) \to \mathbb{Z}\bigl[u, v, u^{-1}, v^{-1}\bigr] ,
\]
see, e.g., \cite[Examples~3.2]{PS08}. In particular, for any variety~$X$, its Betti numbers are encoded in~${\rm \bf e} ( \mathfrak{e} ([X]) )$.

We will need the following result, which we prove in Section \ref{proofHGquot}
(for the case where $X$ is smooth, see, e.g., \cite[Proposition~4.3]{Flo}).
\begin{Proposition}\label{HGquot}
Let $X$ be a quasi-projective variety with an action of a finite group $G$.
Then
\[
\mathfrak{e}([X/G]) = \sum_{i\geq 0}(-1)^i \bigl[ \bigl( H^i_{\rm c} (X, \mathbb{Q}), W_\bullet , F^\bullet \bigr)^G \bigr] ,
\]
where
\[
\bigl( H^i_c (X, \mathbb{Q}), W_\bullet , F^\bullet \bigr)^G =
\frac{1}{|G|}\sum_{g\in G}g_*\bigl( H^i_c (X, \mathbb{Q}), W_\bullet , F^\bullet \bigr)
\]
is the $G$-invariant part of $\bigl( H^i_c (X, \mathbb{Q}), W_\bullet , F^\bullet \bigr)$.
\end{Proposition}

Applying $\mathfrak{e}$ to \eqref{maineq_var}, we obtain
\begin{align*}
\mathfrak{e} \bigl( \bigl[\overline{\mc{B}}\bigr] \bigr) &= \mathfrak{e} \bigl( [\mc{B}] \circ \bigl( [I_1] + D\bigl[\overline{\mc{B}}\bigr] \bigr) \bigr)+
\mathfrak{e} \bigl( [I_2] \circ D[\overline{\mc{B}}] \bigr) -
\mathfrak{e} \bigl( [ (I_2)_2 \times_{B^2} (D\overline{\mc{B}} \ast D\overline{\mc{B}}) ] \bigr) .
\end{align*}
Using \eqref{main_eq__var_1st+_2},
\eqref{main_eq_var_2nd+} and \eqref{main_eq_var_3rd+},
we rewrite the part in degree $n$ of the right-hand side
of the previous equation as follows:
\begin{align*}
& \sum_{m=3}^n \mathfrak{e} \Biggl( \Biggl[ \Biggl(
\coprod_{\boldsymbol{c} \in {\overline{I}_{\boldsymbol{\mu}}(\BG)}^m}
\coprod_{\substack{\underline{k} \in \N^m \\ |\underline{k}|=n}}
 \mc{B}_{\boldsymbol{c}} \times
{\rm Sh}(\underline{k}) \times_{i=1}^m \Bigl( (\{ c_i \})_{k_i} \coprod \overline{\mc{B}}_{k_i, \iota (c_i)}\Bigr) \Biggr) \Big/\boldsymbol{S}_m \Biggr] \Biggr) \\
& \qquad{}+
\mathfrak{e} \Biggl(
\Biggl[ \Biggl( \coprod_{c \in {\overline{I}_{\boldsymbol{\mu}}(\BG)}} \coprod_{\substack{\underline{k} \in \N^2 \\ |\underline{k}|=n}}
{\rm Sh}(\underline{k}) \times \overline{\mc{B}}_{k_1,c} \times \overline{\mc{B}}_{k_2, \iota (c)} \Biggr) \Big/\boldsymbol{S}_2 \Biggr]
\Biggr) \\
& \qquad{}- \sum_{c \in {\overline{I}_{\boldsymbol{\mu}}(\BG)}} \sum_{\substack{\underline{k} \in \N^2 \\ |\underline{k}|=n}}
\mathfrak{e} \bigl( \bigl[ {\rm Sh}(\underline{k}) \times \overline{\mc{B}}_{k_1,c} \times \overline{\mc{B}}_{k_2, \iota (c)} \bigr] \bigr) .
\end{align*}
Using Proposition \ref{HGquot}, we deduce the following expressions for the three summands above:
the class of the $\boldsymbol{S}_m$-invariant part of
\begin{gather*}
 \sum_{\substack{3\leq m \leq n \\ s, t\geq 0}} (-1)^{s+t} \Biggl(
\bigoplus_{\substack{\boldsymbol{c} \in \overline{I}_{\boldsymbol{\mu}}^m \\
\underline{k} \in \N^m \\ |\underline{k}|=n}}
H^s_{\rm c} ( \mc{B}_{\boldsymbol{c}}) \otimes
H^0({\rm Sh}(\underline{k})) \otimes H^t
\Bigl( \times_{i=1}^m ( (\{ c_i \})_{k_i} \coprod \overline{\mc{B}}_{k_i, \iota (c_i)})\Bigr) \Biggr)
\end{gather*}
for the first summand, where ${\overline{I}_{\boldsymbol{\mu}}} = {\overline{I}_{\boldsymbol{\mu}}(\BG)}$;
\begin{gather*}
\sum_{s, t \geq 0} (-1)^{s+t}
\Biggl[ \Biggl( \bigoplus_{\substack{c \in {\overline{I}_{\boldsymbol{\mu}}(\BG)}\\
\underline{k} \in \N^2 , |\underline{k}|=n}}
H^0({\rm Sh}(\underline{k})) \otimes H^s\bigl(\overline{\mc{B}}_{k_1,c}\bigr)
\otimes H^t\bigl(\overline{\mc{B}}_{k_2, \iota (c)} \bigr) \Biggr)^{\boldsymbol{S}_2} \Biggr]
\end{gather*}
for the second, and
\begin{gather*}
- \sum_{c \in {\overline{I}_{\boldsymbol{\mu}}(\BG)}} \sum_{\substack{\underline{k} \in \N^2 \\ |\underline{k}|=n}}
\sum_{s, t \geq 0} (-1)^{s+t} \bigl[ H^0( {\rm Sh}(\underline{k})) \otimes H^s\bigl(\overline{\mc{B}}_{k_1,c}\bigr) \otimes
H^t\bigl(\overline{\mc{B}}_{k_2, \iota (c)} \bigr) \bigr]
\end{gather*}
for the third one.
\begin{Example}
If $n=3$, $\overline{\mc{B}}_3 = \mc{B}_3$. As we will explain at the beginning of Section~\ref{abelian case},
$\mc{B}_3$ is a~finite set, which is in bijection (after the choice of a geometric basis of the fundamental group
$\pi_1\bigl(\mathbb{P}^1 \setminus \{ 0, 1, \infty \}, p_0\bigr)$) with the set
\[
\bigl\{ (g_1, g_2, g_3) \in G^3 \mid g_1 \cdot g_2 \cdot g_3 =1 \bigr\}\big/G ,
\]
where $G$ acts via simultaneous conjugation.

Let us now consider the case where $n=4$. In this case we have $m=4$ and $m=3$ in the formula above.

When $m=4$, necessarily $k_1 = \dots = k_m=1$, hence $\overline{\mc{B}}_{k_i,c_i}=\varnothing$,
$( \{ c_i \})_{k_i}=\operatorname{Spec} (\mathbb{C})$, ${\rm Sh}(\underline{k}) = \boldsymbol{S}_m$, and the first summand
yields
\[
\sum_{s\geq 0} (-1)^{s} \bigl[
\oplus_{\boldsymbol{c} \in {\overline{I}_{\boldsymbol{\mu}_r}(\BG)}^4}
H^s_{\rm c} ( \mc{B}_{\boldsymbol{c}}) \bigr] = \sum_{s\geq 0} (-1)^{s} \bigl[
H^s_{\rm c} ( \mc{B}_4) \bigr] .
\]
Notice that the (tensor) factor $H^0({\rm Sh}(\underline{k})) =H^0( \boldsymbol{S}_m)$ cancels when we take the
$\boldsymbol{S}_m$-invariants.

When $m=3$, $\underline{k}=(2, 1, 1), (1, 2, 1), (1, 1, 2)$, $\overline{\mc{B}}_{2, \iota (c_i)} = {\mc{B}}_{2, \iota (c_i)}$,
while $\overline{\mc{B}}_{1, \iota (c_i)} = \varnothing$,
therefore the first summand yields the class of the $\boldsymbol{S}_3$-invariant part of
\begin{align*}
&\oplus_{\boldsymbol{c} \in {\overline{I}_{\boldsymbol{\mu}}(\BG)}^3}
H^0 ( \mc{B}_{\boldsymbol{c}}) \otimes
H^0({\rm Sh}(2,1,1)) \otimes H^0
( {\mc{B}}_{2, \iota (c_1)}) \\
&\qquad \oplus_{\boldsymbol{c} \in {\overline{I}_{\boldsymbol{\mu}}(\BG)}^3}
H^0 ( \mc{B}_{\boldsymbol{c}}) \otimes
H^0({\rm Sh}(1,2,1)) \otimes H^0
( {\mc{B}}_{2, \iota (c_2)}) \\
&\qquad \oplus_{\boldsymbol{c} \in {\overline{I}_{\boldsymbol{\mu}}(\BG)}^3} H^0 ( \mc{B}_{\boldsymbol{c}}) \otimes
H^0({\rm Sh}(1,1,2)) \otimes H^0
( {\mc{B}}_{2, \iota (c_3)}) .
\end{align*}
From the second summand, we have contributions only when $\underline{k}=(2,2)$, in this case we get
\[
\bigl[ \bigl( \oplus_{c \in {\overline{I}_{\boldsymbol{\mu}}(\BG)}}
H^0({\rm Sh}(\underline{k})) \otimes H^0({\mc{B}}_{2,c})
\otimes H^0({\mc{B}}_{2, \iota (c)} ) \bigr)^{\boldsymbol{S}_2} \bigr] .
\]
The third summand is
\[
\bigl[ \oplus_{c \in {\overline{I}_{\boldsymbol{\mu}}(\BG)}}
H^0({\rm Sh}(\underline{k})) \otimes H^0({\mc{B}}_{2,c})
\otimes H^0({\mc{B}}_{2, \iota (c)} )\bigr] .
\]

The previous formulae allow to determine $\mathfrak{e}\bigl(\bigl[\overline{\mathcal{B}}_4\bigr]\bigr) \in K_0^{\boldsymbol{S}_4}({\rm MHS})$
from the class $[\mathcal{B}]$, as an element of $K_0^{\ms S}({\rm Var}/B)$.
Applying again \eqref{maineq_var}, and using
the previous expression for $\mathfrak{e}\bigl(\bigl[\overline{\mathcal{B}}_4\bigr]\bigr)$ we can
compute $\mathfrak{e}\bigl(\bigl[\overline{\mathcal{B}}_5\bigr]\bigr) \in K_0^{\boldsymbol{S}_5}({\rm MHS})$
in terms of $[\mathcal{B}] \in K_0^{\ms S}({\rm Var}/B)$. Proceeding in this way one can
compute $\mathfrak{e}\bigl(\bigl[\overline{\mathcal{B}}_n\bigr]\bigr)$, for any $n$, in terms of $[\mathcal{B}] \in K_0^{\ms S}({\rm Var}/B)$.

So the problem of determining $\mathfrak{e}\bigl(\bigl[\overline{\mathcal{B}}\bigr]\bigr)$ is reduced to the
knowledge of $[\mathcal{B}] \in K_0^{\ms S}({\rm Var}/B)$. In Section~\ref{abelian case},
we will discuss a further simplification that occurs in the case where $G$ is abelian.
\end{Example}

\subsection{Proof of Proposition \ref{HGquot}}\label{proofHGquot}
We use the fact that the HG-characteristic factors through the Grothendieck group of motives \cite[Appendix~A]{P09}
and that the motive of $X/G$ is isomorphic to the $G$-invariant part of the motive of $X$ \cite[Corollary~5.3]{dBN98}.
Here we follow the presentation given in \cite{dBN98,GS96}, \cite[Section~2]{G01}
and we refer to these articles (and to the references therein) for a more complete treatment.

We denote with $\mathbf{V}$ the category of smooth projective varieties over $\mathbb{C}$
(here we avoid the most general situation because we don't need it,
but notice that one can replace $\mathbb{C}$ with any algebraically closed field of characteristic $0$,
and for some purposes even with any field of characteristic $0$).
To define the category of (pure effective) Chow motives over $\mathbb{C}$, $\mathbf{M}$,
one first consider the category $\mathbf{C}$ of correspondences.
The objects of $\mathbf{C}$ are the same as those of $\mathbf{V}$, morphisms are defined as follows:
\[
\operatorname{Hom}_{\mathbf{C}}(X,Y) := A^{\dim X}(X\times Y) ,
\]
where $A^n$ denotes the Chow group of algebraic cycles of codimension $n$ modulo rational equivalence
tensored by $\mathbb{Q}$. It is possible to define a composition of morphisms such that $\mathbf{C}$
becomes an additive tensor category.
The category~$\mathbf{M}$ is the pseudo-abelian closure of $\mathbf{C}$.
An object of~$\mathbf{M}$ is a~pair~$(X, p)$, where~$X$ is a smooth projective variety and $p\in A^{\dim X}(X\times X)$
satisfies $p^2 =p$.
There is a natural functor $\mathbf{V} \to \mathbf{M}$ that sends~$X$ to~$(X, [\Delta_X])$,
where $\Delta_X \subseteq X\times X$ is the class of the diagonal.
Disjoint union and product of varieties can be extended to~$\mathbf{M}$, where they are denoted by $\oplus$
and $\otimes$, respectively.

The Grothendieck group $K_0(\mathbf{M})$ is the quotient of the free abelian group on the isomorphism classes $[M]$
of objects $M$ in $\mathbf{M}$ by the subgroup generated by elements of the form
$[M]-[M']-[M'']$ whenever $M\cong M'\oplus M''$.
The operation $\otimes$ induces a ring structure on $K_0(\mathbf{M})$.

The following result holds true over any field of characteristic $0$ (see \cite[Theorem~4]{GS96} and~\cite{GN02}).
\begin{Theorem}
There exists a unique ring homomorphism $\overline{h} \colon K_0({\rm Var}) \to K_0(\mathbf{M})$
such that $\overline{h}([X])=[(X, [\Delta_X])]$ for $X$ smooth and projective.
\end{Theorem}
Let us recall for later use that $\overline{h}([X])$ coincides with the class in $K_0({\rm Hot}^b (\mathbf{M}))$
 of the weight complex $W(X)$ of $X$ \cite[Sections~2.1 and~3.2]{GS96}.
${\rm Hot}^b (\mathbf{M})$ denotes the category of bounded chain complexes in
$\mathbf{M}$ up to homotopy, and one uses the identification of $K_0({\rm Hot}^b (\mathbf{M}))$ with $K_0(\mathbf{M})$
to get a class in~$K_0(\mathbf{M})$.

The cohomology of a motive is defined as follows. For any object $M=(X, p)$ of $\mathbf{M}$
\[
H^i(M):= {\rm Im}\bigl(p_* \colon H^i(X, \mathbb{Q}) \to H^i(X, \mathbb{Q})\bigr) ,
\]
where $p_*$ denotes the action induced on cohomology by $p$,
$p_* (\alpha):= ({\rm pr}_2)_* (p\cdot {\rm pr}_1^*(\alpha))$, where ${\rm pr}_i \colon X\times X \to X$
is the projection to the $i$-th component.
Since $p_*$ is a morphism in ${\rm MHS}$, $H^i(M)$ is an object of ${\rm MHS}$,
and following \cite[Appendix A-3]{P09} we define
\[
\mathfrak{e}_{\mathbf{M}} (M):= \sum_{i \geq 0} (-1)^i \bigl[H^i(M)\bigr] \in K_0({\rm MHS}) .
\]
Notice that $\mathfrak{e}_{\mathbf{M}}$ is additive on the objects of $\mathbf{M}$, therefore
it induces a group homomorphism
\[
\mathfrak{e}_{\mathbf{M}} \colon \ K_0(\mathbf{M}) \to K_0({\rm MHS}).
\]
It is proved in \cite[Appendix A]{P09} that the HG-characteristic lifts to motives, i.e., \[
 \mathfrak{e} = \mathfrak{e}_{\mathbf{M}} \circ \overline{h} .
\]
It follows from this that $\mathfrak{e}([X/G]) = \mathfrak{e}_{\mathbf{M}}\bigl( \overline{h} ([X/G])\bigr)$.

Corollary 5.3 in~\cite{dBN98} implies that there is an isomorphism of weight complexes,
$W(X/G) \cong W(X)^G$, where $W(X)^G = \bigl(W(X), \frac{1}{|G|}\sum_{g \in G}[g]\bigr)$
is the image of $W(X)$ under the projector $\frac{1}{|G|}\sum_{g \in G}[g]$.
Here $W(X)$ has to be considered as an object of the category of functors from $G$ to ${\rm Hot}^b (\mathbf{M})$.

Let us now apply the functor $H^n$ to $W(X/G)$. We obtain a complex of
$\mathbb{Q}$-vector spaces whose $i$-th cohomology group is
denoted $R^iH^n(W(X/G))$, and it coincides with
${\rm gr}_n^W H_{\rm c}^{i+n}(X/G, \mathbb{Q})$ \cite[Section~3.1]{GS96}.
On the other hand,
\begin{align*}
R^iH^n(W(X/G)) &= R^iH^n\bigl(W(X)^G\bigr) = R^i ( H^n (W(X)) )^G \\
&= \bigl( R^i H^n (W(X))\bigr)^G = \bigl( {\rm gr}_n^W H_{\rm c}^{i+n}(X, \mathbb{Q}) \bigr)^G ,
\end{align*}
where we have used the fact that $H^n(X,p)= {\rm Im} (p_* \colon H^n (X) \to H^n(X))$ for any motive $(X, p)$,
by definition, and that taking $G$-invariants is exact for vector spaces over $\mathbb{Q}$.

The claim follows because the Grothendieck group of mixed Hodge structures
is isomorphic to the one of pure Hodge structures (${\rm HS}$) via the homomorphism $K_0({\rm MHS}) \to K_0({\rm HS})$
defined by sending $[(V, W_\bullet, F^\bullet)]$ to $\sum_{i\in \mathbb{Z}} \bigl[{\rm gr}_i^W V\bigr]$ \cite[Corollary~3.9]{PS08}.

\subsection{Application: the abelian case}\label{abelian case}
In this section, we consider the forgetful morphism $\pi \colon \balsm \to \mathcal{M}_{0,n}$
that sends any twisted $G$-cover $(\mathcal{C}\to C, \Sigma, f \colon \mathcal{C} \to \BG)$
of an $n$-pointed curve of genus $0$
to $(C, |\Sigma_1|, \dots , |\Sigma_n|) =: (C, p_1, \dots , p_n)$
(we adopt the notation of Definition~\ref{twiste_Gcover_Npointed_curve}).
As in the previous section $\balsm$ and $\mathcal{M}_{0,n}$ denote coarse moduli spaces.
The morphism $\pi$ is finite and \'{e}tale (see, e.g., \cite[Lemma~5.2.2]{ACV03}), and its fiber
over $(C, p_1, \dots , p_n)$ can be identified with the quotient set
\[
\operatorname{Hom} (\pi_1 (C\setminus \{ p_1, \dots , p_n \} , p_0), G)/G ,
\]
where $p_0 \in C\setminus \{ p_1, \dots , p_n \}$ is a base point and $G$ acts by conjugation
(see, e.g., \cite{BCK, Fulton69}).

Let us fix a geometric basis of $\pi_1 (C\setminus \{ p_1, \dots , p_n \} , p_0)$,
i.e., loops $\gamma_1, \dots , \gamma_n$ based at $p_0$, that go near $p_1, \dots , p_n$
along given paths, turn around the marked points and go back to $p_0$ along the same paths.
Furthermore, $\gamma_1, \dots , \gamma_n$ intersect each other only at $p_0$,
their homotopy classes $[\gamma_1], \dots , [\gamma_n]$ generate the fundamental group
$\pi_1 (C\setminus \{ p_1, \dots , p_n \} , p_0)$ and are subject only to the relation
$[\gamma_1] \cdots [\gamma_n] = 1$.

Using this geometric basis, we can identify the set of morphisms $\operatorname{Hom} (\pi_1 (C\setminus \{ p_1, \dots , p_n \} ,\allowbreak p_0), G)$
with the set ${\rm HV}(G, n)$ of $G$-Hurwitz vectors in $G$ of genus $0$ (we use the same notation as in
\cite[Definition~2.1]{CLP16}),
\[
{\rm HV}(G, n) := \bigl\{ (g_1, \dots , g_n) \in G^n \mid g_1 \cdot \dots \cdot g_n=1 \bigr\} ,
\]
the action by $G$ corresponds to the simultaneous conjugation of $g_1, \dots , g_n$.
The Artin's braid group $Br_n$ is generated by the so-called elementary braid $\sigma_1, \dots , \sigma_{n-1}$
and it acts on ${\rm HV}(G, n)$ as follows:{\samepage
\[
(g_1, \dots , g_n) \cdot \sigma_i = \bigl(g_1, \dots , g_{i-1}, g_i g_{i+1} g_i^{-1}, g_i, g_{i+2}, \dots , g_n\bigr) .
\]
Notice that this action commutes with the $G$-action, hence it induces an action on ${\rm HV}(G, n)/G$.}

One can use this action to describe the connected components of $\balsm$.
In particular, when $G$ is abelian, by \cite[Proposition~2.1]{BCK}, we have that
$\pi \colon \balsm \to \mathcal{M}_{0,n}$ is a trivial bundle, hence $\balsm \cong \mathcal{M}_{0,n} \times {\rm HV}(G, n)$
and the evaluation maps correspond to the projections ${\rm HV}(G, n) \to G$.
Therefore, we obtain an expression of $[\mathcal{B}] \in K_0^{\ms S}({\rm Var}/B)$ in terms of
the classes $[\mathcal{M}_{0,n}]$, for all $n$.
Using the known results about the cohomology of $\mathcal{M}_{0,n}$
(e.g., \cite{Get95,Kee92, Manin95}), one can determine the Betti numbers of $\mathcal{B}_n$,
for every $n$, hence those of $\overline{\mathcal{B}}_n$ using the results of the previous section.

\subsection*{Acknowledgments}
This article is dedicated to Lothar G\"ottsche, in occasion of his 60th birthday, with esteem and admiration.
We would like to thank the anonymous referees for their precious comments and observations
about an earlier version of the article, and the guest editors of the special volume
for the invitation and for their comments.
The second author was partially supported by the national project
2017SSNZAW 005-PE1 ``Moduli Theory and Birational Classification'',
by the research group GNSAGA of INDAM and by FRA of the University of Trieste.

\pdfbookmark[1]{References}{ref}
\LastPageEnding

\end{document}